\documentclass[reqno,11pt]{amsart}
\usepackage{latexsym,amssymb,amsfonts,amsmath,amsthm}
\usepackage[usenames]{color}
\usepackage{eucal}
\usepackage{graphicx}

\newtheorem{thm}{Theorem}[section]

\newtheorem{lem}[thm]{Lemma}

\newtheorem{prop}[thm]{Proposition}
\theoremstyle{definition}

\theoremstyle{remark}
\newtheorem{rem}[thm]{Remark}
\numberwithin{equation}{section}
\newcommand{\Real}{\mathbb R}
\newcommand{\Z}{\mathbb Z}
\newcommand{\eps}{\varepsilon}

\newcommand{\F}{\mathcal{F}}

\newcommand{\one}[1]{\mathbf{1}_{\{#1\}}}

\renewcommand{\P}{\mathbb{P}}

\newcommand{\E}{\mathbb{E}}

\DeclareMathOperator*{\var}{var}
\DeclareMathOperator*{\cov}{cov}

\newcommand{\iunit}{\mathbf{i}}
\newcommand{\bxi}{\boldsymbol\xi}
\newcommand{\bzeta}{\boldsymbol\zeta}
\newcommand{\bgamma}{\boldsymbol\gamma}

\definecolor{LightGray}{rgb}{0.9,0.9,0.9}

\begin{document}

\title[]{Estimation in threshold autoregressive models with correlated innovations}%
\author{P. Chigansky*}%
\thanks{* the corresponding author}
\address{Department of Statistics,
The Hebrew University,
Mount Scopus, Jerusalem 91905,
Israel}
\email{pchiga@mscc.huji.ac.il}

\author{Yu. A. Kutoyants}%
\address{Laboratoire de Statistique et Processus,
Universite du Maine,
France}
\email{kutoyants@univ-lemans.fr}


\begin{abstract}
Large sample statistical analysis of threshold autoregressive (TAR) models is usually based on the assumption
that the underlying driving noise is uncorrelated. In this paper, we consider a model, driven by Gaussian  noise with
geometric correlation tail and derive a complete characterization of the asymptotic distribution for the Bayes estimator of the
threshold parameter.
\end{abstract}
\maketitle

\section{Introduction: the setting and the main result}

Let $(X_j)_{j\in \Z_+}$ be the sequence generated by the recursion
\begin{equation}
\label{AR}
X_j = \Big(\rho^+ \one{X_{j-1}\ge \theta} + \rho^- \one{X_{j-1} <\theta}\Big)X_{j-1} + \epsilon_j,\quad j\ge 1
\end{equation}
where $(\epsilon_j)_{j\in \Z_+}$ is a random process
with known distribution,  $\rho^+$ and $\rho^-$ are known constants  and $\theta$ is the unknown
threshold parameter to be estimated from the sample $X^n:=(X_1,...,X_n)$. The equation \eqref{AR} is a basic instance of the threshold autoregression (TAR) models,
which play a considerable role in the theory and practice of time series. This type of models have been studied by
statisticians for already more than three decades, producing  interesting theory and finding many important applications,
some of which can be traced in the early and more recent surveys \cite{Tong83} and \cite{Tong11}, \cite{Tsay89}, \cite{H11}, \cite{ChK10}, \cite{Kut10} (see also, e.g., \cite{LT05}, \cite{LTL07} for the analysis of the related  moving average threshold (TMA) models).

When it comes to the large sample asymptotic analysis of the estimators, the standard conditions imposed on the models such as \eqref{AR}
are

\medskip

\begin{enumerate}\addtolength{\itemsep}{0.5\baselineskip}
\renewcommand{\theenumi}{\roman{enumi}}
\item\label{ci} strong ergodicity of the observed process $(X_j)$

\item\label{cii} independence of the driving random variables $\epsilon_j$'s
\end{enumerate}

\medskip

Departure from these assumptions poses challenging problems. For the model \eqref{AR}, the condition \eqref{ci} fails if the absolute value of either $\rho^+$ or $\rho^-$
is greater or equal to 1. If the process $(X_j)$ is null recurrent (e.g. $\rho_+=1$ and $|\rho_-|<1$), characterization of the exact large sample asymptotic distribution of the likelihood based estimators of the threshold parameter $\theta$ remains an open problem (see Remark \eqref{rem:1.2} below). If $\theta\ne 0$ is assumed to be known, the asymptotic distribution of the coefficients' estimators in the non-ergodic case has been studied in \cite{PCT91}, \cite{MH01}, \cite{LLS11}.

TAR models beyond  the independence assumption \eqref{cii} of the driving noise sequence has not yet been addressed.
As we shall shortly see, in the dependent case the problem falls into the framework of
statistical inference of hidden Markov models (HMM), where the driving noise plays the role of the hidden signal
(see Ch. 10-12 in \cite{CMR05} and the references therein).
However, most of the HMM literature deals with  locally asymptotically normal (LAN) experiments and,
to the best of our knowledge, non-LAN models with partial observations
have not yet been studied systematically.

In this paper, we consider the model \eqref{AR} in which $(\epsilon_j)$ is a sequence with geometrically decaying
correlation. More precisely, let $X=(X_j)_{j\in \Z_+}$ be generated by the recursion
\begin{equation}
\label{ar}
X_j = \Big(\rho^+ \one{X_{j-1}\ge \theta} + \rho^- \one{X_{j-1} <\theta}\Big)X_{j-1} + \xi_{j-1} + \eps_j,
\end{equation}
subject to $X_0\sim N(0,1)$, where $\rho:=|\rho^+|\vee |\rho^-|<1$ and the unknown parameter $\theta$ takes values in an open bounded
subset $\Theta\subset \Real$.
We shall consider the problem with discontinuous drift function $f(x, \theta):= \big(\rho^+  \one{x\ge \theta} +\rho^- \one{x<\theta}\big)x$, and thus
assume $\rho^+\ne \rho^-$ and  $0\not\in \Theta$.
The driving noises $(\eps_j)_{j\in \Z_+}$ and $(\xi_j)_{j\in  \Z_+}$ are assumed  independent:
the {\em white noise} component $(\eps_j)$ is a sequence of i.i.d. $N(0,1)$ random variables and
the {\em colored noise}  $(\xi_j)$ is the Gaussian AR(1) process, generated by the linear recursion
\begin{equation}
\label{xi}
\xi_j = a \xi_{j-1} + \zeta_j, \quad j\ge 1,
\end{equation}
where   $(\zeta_j)_{j\in \Z_+}$ are i.i.d. $N(0,1)$ random variables and $a$ is a known constant
$|a|<1$, controlling the bandwidth of the noise.

All the aforementioned random variables are defined on a measurable space $\big(\Omega, \F\big)$,
with the family of   probabilities $\big(\P_\theta\big)_{\theta\in \Theta}$, indexed by the unknown parameter.
For integers $k\ge m$, we define $\F_{k,m}:=\sigma\{\zeta_i,\eps_i\; k\le i\le m\}$ and set $\F_m:=\F_{0,m}$ and
$\F_{k,\infty}:=\bigvee_{i\ge k}\F_{k,i}$. All the processes in our problem are adapted to the filtration $(\F_j)_{j\in \Z_+}$
and we shall assume that $\F=\F_{0,\infty}$.
Finally we define the observed filtration $\F^X_j:= \sigma\{X_i,\; i\le j\}\subset \F_j$.

The recursions \eqref{ar} and \eqref{xi} form a conditionally Gaussian system, which means that the conditional law of
$\xi_n$ given $X^n$ is Gaussian, and by Theorem 13.5 in \cite{LiSh2}
\begin{equation}
\label{inno}
X_j = f(X_{j-1}, \theta) + \widehat \xi_{j-1}(\theta) + \sqrt{1+\gamma_{j-1}}\widehat \eps_j,
\end{equation}
where
$$
\widehat \eps_j : = \frac{1}{\sqrt{1+\gamma_{j-1}}}\Big(\xi_{j-1}-\widehat \xi_{j-1}(\theta) +\eps_j\Big), \quad j\ge 1
$$
is the innovation sequence of i.i.d. $N(0,1)$ random variables.
The process
$\widehat \xi_j(\theta) := \E_\theta\big(\xi_j|\F^X_j\big)$  and the deterministic sequence
$\gamma_j := \E_\theta\big(\xi_j-\widehat \xi_j(\theta)\big)^2$ satisfy the generalized Kalman filter equations
\begin{align}\label{KF1}
\widehat \xi_j(\theta) &= a \widehat \xi_{j-1}(\theta) +
\frac{a \gamma_{j-1}}{1+\gamma_{j-1}}\Big(X_j - f(X_{j-1},\theta)-\widehat \xi_{j-1}(\theta)\Big) \\
\label{KF2}
\gamma_j &= a^2\gamma_{j-1} + 1 -\frac{a^2 \gamma^2_{j-1}}{1+\gamma_{j-1}},
\end{align}
subject to $\widehat \xi_0=0$ and $\gamma_0=\var(\xi_0)$.

To avoid inessential technicalities, we shall assume that $\xi_0$ and $X_0$ are independent and
$\xi_0\sim N(0, \gamma)$, where $\gamma$ is the unique positive root of the equation
$$
\gamma = a^2\gamma + 1 -\frac{a^2 \gamma^2}{1+\gamma}.
$$
In this case, the conditional mean $\widehat \xi_j(\theta)$ satisfies \eqref{KF1} with constant coefficients:
\begin{equation}
\label{KF1c}
\widehat \xi_j(\theta) = a \widehat \xi_{j-1}(\theta) +
\frac{a \gamma}{1+\gamma}\Big(X_j - f(X_{j-1},\theta)-\widehat \xi_{j-1}(\theta)\Big).
\end{equation}
It can be seen that $(\gamma_j)$ converges to $\gamma$ exponentially fast and all the results claimed below hold for $\xi_0$ with
an arbitrary Gaussian distribution.

Let $\bar X_0:=X_0$ and $\bar X_j := X_j-f(X_{j-1},\theta)$, $j\ge 1$ and note that $\F^{\bar X}_{j}= \F^{X}_{j}$ for all $j\ge 0$.
By the definition of the conditional expectation, $\xi_{j-1}-\widehat\xi_{j-1}$ is orthogonal to $\F^{X}_{j-1}$, and thus to $\F^{\bar X}_{j-1}$.
Moreover, since the process $(\bar X_j, \xi_j,\widehat\xi_j)$ is Gaussian, $\xi_{j-1}-\widehat\xi_{j-1}$ is independent of
$\F^{\bar X}_{j-1}$ and thus of $\F^X_{j-1}$ as well. Further, since
$\sqrt{1+\gamma}\widehat \eps_j = \xi_{j-1}-\widehat\xi_{j-1}+\eps_j$ and $\eps_j$ is independent of $\F_{j-1}$,
independence of $\widehat \eps_j$ and $\F^X_{j-1}$ follows and  the  representation \eqref{inno}
implies that the likelihood of the data $X^n$ is given by
\begin{multline}
\label{lh}
L_n(X^n;\theta) = \frac{1}{\sqrt{2\pi}}\exp\Big(-\frac 1  2 X_0^2\Big)\times \\\left(\frac 1{\sqrt{2\pi(1+\gamma)}}\right)^n
\exp\left(
-\frac 1 2 \frac 1 {{1+\gamma}}\sum_{j=1}^n \Big(X_j-f(X_{j-1},\theta) - \widehat \xi_{j-1}(\theta)\Big)^2
\right)
\end{multline}
The likelihood function is discontinuous in $\theta$ and hence we are faced with an irregular
statistical experiment. In such problems, the maximum likelihood estimator is often asymptotically
inferior to the Bayes estimator $\widetilde \theta_n$, while the latter is typically asymptotically efficient for arbitrary
continuous positive prior densities in the following minimax sense (see Theorem 9.1, \cite{IH}):
$$
\lim_{\delta\to 0}\varliminf_{n\to\infty}\inf_{T_n}\sup_{\theta:|\theta-\theta_0|\le \delta}n^2\E_{\theta}\big(T_n - \theta\big)^2\ge
\lim_{\delta\to 0}\varliminf_{n\to\infty}\sup_{\theta:|\theta-\theta_0|\le \delta}n^2\E_{\theta}\big(\widetilde \theta_n - \theta\big)^2,
$$
where $T_n$'s are $\F^X_n$-measurable statistics.

The Bayes estimator for the problem at hand has relatively low computational complexity,
since the likelihood function is piecewise constant in $\theta$ and has at most $n$ jumps at
$\{X_0,...,X_{n-1}\}$. More precisely, for a prior density  $\pi$, the Bayes estimator with respect to the
quadratic risk is given by
\begin{equation}\label{BE}
\widetilde \theta_n :=
\frac{\int_\Theta \theta L_n(X^n;\theta)\pi(\theta)d\theta}{\int_\Theta L_n(X^n;\theta)\pi(\theta)d\theta}=
\frac
{\sum_{j: \{X_{(j-1)}, X_{(j)}\}\cap \Theta\ne \emptyset} L_n\big(X^n;X_{(j)}\big) \int_{X_{(j-1)}}^{X_{(j)}}\theta \pi(\theta)d\theta}
{\sum_{j: \{X_{(j-1)}, X_{(j)}\}\cap \Theta\ne \emptyset} L_n\big(X^n;X_{(j)}\big) \int_{X_{(j-1)}}^{X_{(j)}} \pi(\theta)d\theta},
\end{equation}
where $X_{(j)}$ is the $j$-th order statistic of $X^{n}$.  If the prior $\pi$ is chosen so that numerical integration in the right hand side is avoided, the computation of $\widetilde \theta_n$ can be carried out in polynomial  time of order $O(n^2)$.

\medskip

The following property, whose proof is deferred to Appendix \ref{app-sec} (see Lemma \ref{lem-erg}), plays a crucial role in the forthcoming analysis
\begin{prop}\label{prop1}
Assume
\begin{equation}
\label{AE}\tag{E}
|\rho^\pm|<1\quad \text{and}\quad |a|<1,
\end{equation}
then the Markov process $(X_j,\xi_j)$ is geometrically ergodic under $\P_\theta$, $\theta \in \Theta$,  with the unique invariant probability density $p(x,y;\theta)$.
\end{prop}

\begin{rem}\label{rem:1.2}
Obviously,  recurrence of $(X_j)$ is necessary for consistent estimation of the threshold parameter. The condition \eqref{AE}
guarantees positive recurrence of the process $(X_j)$, which is an essential ingredient in derivation of the large sample asymptotic
of Theorem \ref{thm} below. In the null recurrent case, i.e. when one of the coefficients have unit absolute value, consistent estimation of $\theta$
seems to be possible and some preliminary calculations show that the corresponding rate may depend on the distribution tail  of the driving noise.
The exact characterization of the large sample asymptotic in this setting remains an open problem, even for independent innovations.
\end{rem}

\medskip

The main result of this paper is the following characterization of the asymptotic distribution of the sequence of Bayes estimators:

\begin{thm}\label{thm}
Let $(\widetilde \theta_n)$ be the sequence of the Bayes estimators with respect to the
quadratic loss function and a prior with continuous positive density $\pi$.
Then for any continuous function $\phi$ with at most polynomial growth
$$
\lim_n \E_{\theta_0} \phi\left(n\big(\widetilde \theta_n-\theta_0\big)\right) = \E_{\theta_0} \phi(\tilde u),
$$
uniformly on compacts from $\Theta$, where
$$
\tilde u = \frac{\int_\Real u Z(u)du}{\int_\Real Z(u)du}
$$
and $\ln Z(u)$, $u\in \Real$ is the following two sided compound Poisson process:
\begin{equation}
\label{Y}
\ln Z(u) = \left\{
\begin{array}{ll}
 \sum_{j=1}^{\Pi^+(u)} \Big(  \beta \eps^+_{j}-\frac 1 2 \beta^2\Big) & u\ge 0 \\
 &\\
 \sum_{j=1}^{\Pi^-(|u|)}  \Big(  \beta \eps^-_{j}-\frac 1 2 \beta^2\Big)   & u < 0
\end{array}
\right.
\end{equation}
Here $\Pi^{+}$, $\Pi^{-}$ are i.i.d Poisson processes with the intensity
$$
\varpi = \int_\Real p(\theta_0,y;\theta_0)dy,
$$
$p(x,y;\theta_0)$ is the unique invariant probability density of the Markov process $(X_j,\xi_j)$ under
$\P_{\theta_0}$,
$\big(\eps^{\pm }_j\big)$  are i.i.d. $N(0,1)$ random variables, independent of $\Pi^{+}$ and $\Pi^{-}$ and
$$
\beta^2 =
\left(\frac{\theta_0(\rho^+-\rho^-)}{\sqrt{1+\gamma}}\right)^2\left(1+\Big(\frac{a\gamma}{1+\gamma}\Big)^2\sum_{j=0}^\infty\Big(\frac{a}{1+\gamma}\Big)^{2j}\right)=
\theta^2_0(\rho^+-\rho^-)^2
\frac{1+\gamma^3}{(1+\gamma)(1+\gamma^2)}.
$$
\end{thm}

\subsection{Generalizations}
\

\medskip
\noindent
{\bf 1.} If $\xi_{j-1}$ is replaced in \eqref{ar} with $\xi_j$, i.e. if the colored noise component enters without the one-step delay,
the model
$$
X_j = f(X_{j-1},\theta) + a\xi_{j-1}+\zeta_j + \eps_j.
$$
is obtained. In this case, the Kalman filter equations take a slightly different form and the asymptotic analysis can be carried out
exactly as in our setting.

On the other hand, if the white noise component $\eps_j$ is omitted, the observed process satisfies the equation
\begin{multline*}
X_j =  f(X_{j-1},\theta) + \xi_{j} = f(X_{j-1},\theta) + a\xi_{j-1}+\zeta_j = \\
f(X_{j-1},\theta) + a\Big(X_{j-1}-f(X_{j-2},\theta)\Big)+\zeta_j.
\end{multline*}
Being a completely observed system, this model fits the setting of \cite{Ch93} or \cite{ChK11} after a straightforward
modification.

\medskip
\noindent
{\bf 2.} Our method is directly applicable to the models, where the colored noise is generated by a linear multivariate recursion:
\begin{equation}
\label{bxi}
\bxi_j = A \bxi_{j-1} + B\bzeta_j,\quad j\ge 1,
\end{equation}
where $(\bzeta_j)$ are i.i.d. standard Gaussian vectors in $\Real^M$ and $A$ and $B$ are $N\times N$ and $N\times M$ matrices
respectively. In this case, the observed process satisfies the scalar recursion
$$
X_j = \Big(\rho^+ \one{X_{j-1}\ge \theta} + \rho^- \one{X_{j-1} <\theta}\Big)X_{j-1} + C^\top \bxi_{j-1} + \eps_j,
$$
where $C$ is a column vector of size $N$. In this setting, the Kalman filter equations read (cf. \eqref{KF1} and \eqref{KF2})
\begin{align*}
\widehat \bxi_j(\theta) &= A \widehat \bxi_{j-1}(\theta) +
\frac{A \bgamma_{j-1} C}{1+C^\top \bgamma_{j-1}C}
\Big(X_j - f(X_{j-1},\theta)-C^\top \widehat \bxi_{j-1}(\theta)\Big) \\
\bgamma_j &= A\bgamma_{j-1} A^\top + BB^\top -\frac{A\bgamma_{j-1}CC^\top \bgamma_{j-1}A^\top}{1+C^\top \bgamma_{j-1}C},
\end{align*}
subject to $\widehat\bxi_0=0$ and $\bgamma_0=\cov(\xi_0,\xi_0)$.
If  $A$ is a stability matrix, i.e. the absolute values of its eigenvalues are strictly less than 1, and the pair $(A,B)$ is controllable:
$$
\mathrm{rank}\begin{pmatrix}
B & A B & ... & A^{N-1} B
\end{pmatrix}=N,
$$
then the solution of the Riccati equation converges to the matrix $\bgamma$, which is the unique strictly positive definite
root of the corresponding algebraic Riccati equation (see e.g. \cite{LiSh2})
$$
\bgamma = A\bgamma A^\top + BB^\top -\frac{A\bgamma CC^\top \bgamma A^\top}{1+C^\top \bgamma C}.
$$
The statement of Theorem \ref{thm} holds  with the rate
$$
\varpi =\int_{\Real^N} p(\theta_0,y;\theta_0)dy,
$$
where $p(x,y;\theta_0)$ is the invariant density of the process $(\bxi_j,X_j)$, and
$$
\beta^2=
\left(\frac{\theta_0(\rho^+-\rho^-)}{\sqrt{1+C^\top \bgamma C}}\right)^2
\left(1+
\sum_{j=0}^\infty \bigg(
C^\top \Big(A-\frac{A\bgamma CC^\top}{1+C^\top \bgamma C}\Big)^j
\frac{A^\top \bgamma C}{1+C^\top \bgamma C}
\bigg)^2
\right).
$$
The latter formula emerges in the proof of the Lemma \ref{lem-3} below,  with the obvious adjustments to the multivariate setting.

The model \eqref{bxi} incorporates the case of the stationary ARMA$(p,q)$ noise:
$$
\xi_j = -\sum_{k=1}^p a_k \xi_{j-k} + \sum_{\ell=0}^q b_\ell \zeta_{j-\ell},
$$
where $a_1,...,a_p$ and $b_0,...,b_q$ are  constants, such that the roots of the polynomial $a_p z^p +...+a_1z+1$
lie in the open unit disk of the complex plain. The canonical state space representation \eqref{bxi} is obtained through the
usual state augmentation
$$
\bxi_j^\top  := \begin{pmatrix}
\xi_j, &
... &
\xi_{j-p+1,} &
\eps_j, &
... &
\eps_{j-q+1}
\end{pmatrix}^\top\in \Real^{p+q}.
$$

\medskip
\noindent
{\bf 3.} Assuming noises with Gaussian distribution is essential, since in this case the filtering equations for the
conditional density of $\xi_j$ given $\F^X_j$ are finite dimensional and, moreover, the conditional mean
$\widehat \xi_j$  satisfies the linear recursion, whose explicit solution is used on several occasions
through the proof and appears in the expression for $\beta^2$. The result can be extended to more general conditionally
Gaussian models, such as, e.g., higher order TAR with possibly heteroscedastic driving noise.
We expect that for non-Gaussian noise, the limit likelihood will still be a two-sided compound Poisson process,
but no neat closed form expression for $\beta^2$ will be available.

\medskip
\noindent
{\bf 4.}
In principle, our technique is applicable to Gaussian sequences with non-Markov structure, such as fractional noises, etc.
The analysis in this setting is more complicated, depending on the ergodic properties of the processes and the
complexity of the filtering equations (whose linearity will  be intact).

\medskip
\noindent
{\bf 5.} Joint asymptotic analysis of the likelihood based estimators of all the parameters in the model can be in principle carried out
using the same weak convergence approach, used in the proof of Theorem \ref{thm} below (see a brief outline in Section \ref{sec-2.2}).
In this case the likelihood  \eqref{lh} is considered as a function of the four unknown parameters $\rho^+,\rho^-, a\in (-1,1)$ and
$\theta\in \Theta$ and the corresponding normalized likelihood ratios read (cf. \eqref{Znu} below)
$$
Z_n(y,w,v,u) := \frac{
L_n\Big(X^n;
\rho^+_0+\frac{y}{\sqrt{n}},\rho^-_0+\frac {w}{\sqrt{n}},a_0+\frac{v}{\sqrt{n}},\theta_0+\frac{u}{n}
\Big)}
{
L_n\big(X^n;\rho^+_0,\rho^-_0,a_0,\theta_0\big)
},
$$
for a fixed value of the parameter vector $(\rho^+_0,\rho^-_0, a_0,\theta_0)$ and  variables $y,w,v,u$ taking values in
appropriate sets.
Note that the localizing scaling of $\theta$ differs from that of the other parameters, in which the likelihood function is smooth.
It is possible to check the weak convergence of processes
\begin{equation}
\label{Znywvu}
Z_n(y,w,v,u) \stackrel{w}{\Longrightarrow} Z(y,w,v) Z(u), \quad \text{as \ } n\to \infty,
\end{equation}
where $Z(u)$ is the same as in \eqref{Y} and
$$
Z(y,w,v) = \exp \left(
\eta^\top \nu -\frac 1 2 \eta^\top I\eta
\right), \quad \eta:=(y,w,v)\in \Real^3,
$$
with  $\nu\sim N(0,I)$ and the Fisher information matrix $I$, whose
explicit expression is cumbersome. The convergence \eqref{Znywvu} implies the weak convergence of errors for the corresponding Bayes estimators
$$
\Big(\sqrt{n}(\widetilde \rho^+_n -\rho^+_0), \sqrt{n} (\widetilde \rho^-_n-\rho^-_0), \sqrt{n}(\widetilde a_n-a_0), n (\widetilde \theta_n-\theta_0)\Big)
\stackrel{w}{\Longrightarrow}
\big(\tilde y, \tilde w, \tilde v, \tilde u\big),
$$
with  $\tilde u$ as in Theorem \ref{thm} and  zero mean normal vector $(\tilde y, \tilde w, \tilde v)$ with covariance $I^{-1}$, independent of $\tilde u$.
Since the LAN property holds with respect to $(\rho^+,\rho^-,a)$, the corresponding Maximum Likelihood Estimators (MLEs) are
also asymptotically efficient.

\medskip

The proof of Theorem \ref{thm} is given in the next section and some supplementary results, concerning the ergodic properties of the relevant processes, appear in  Appendix \ref{app-sec}. Some simulations, demonstrating the contributions  of this paper, are gathered in Section \ref{sec-sim}.

\section{The proof of Theorem \ref{thm}}\label{sec-pf}

\subsection{The notations}
The actual unknown value of the parameter will be denoted by $\theta_0$ and  will be assumed to belong to a generic compact $K\subset \Theta$.
We shall use $C_i$, $i\in \mathbb{N}$ to denote absolute constants, whose values depend only on the known parameters of the model and the compact $K$ and may change at each appearance.
For random sequences $(x_n)$, $(y_n)$ and a positive real decreasing sequence $(r_n)$,  $x_n=y_n+O(r_n)$ means that
$\sup_n |x_n-y_n|/r_n$ is a random variable with moments, bounded uniformly over $K$.
Throughout, we reserve
$$
b:=\frac a{1+\gamma} \quad \text{and}\quad c:=\frac{a \gamma}{1+\gamma}.
$$
For an integer $n$, the quantities such as $n^{1/2}$ and $n^{1/4}$,
are understood to be rounded to the nearest integer.
For $\ell < k$, we set $\sum_{j=k}^\ell (...) =0$ and $\prod_{j=k}^\ell (...)=1$.
For a vector $z\in \Real^d$, $\|z\|$ stands for the $\ell_1$-norm.
Finally, $\widetilde \P_{\theta_0}$ and $\widetilde \E_{\theta_0}$
denote  the probability on $(\Omega,\F)$ and the corresponding expectation, under which all the processes are stationary
(the unique existence of such  probability is argued in Appendix \ref{app-sec}).

\subsection{Preliminaries}\label{sec-2.2}
Consider the scaled sequence of likelihoods
\begin{equation}\label{Znu}
Z_n(u) := \frac{L_n(X^n;\theta_0+u/n)}{L_n(X^n;\theta_0)}, \quad u\in \mathbb{U}_n:=n(\Theta -\theta_0), \quad n\ge 1.
\end{equation}
The Bayes estimator of $\theta$ is given by
\begin{multline*}
\widetilde \theta_n =
\frac{\int_\Theta \theta L_n(X^n;\theta)\pi(\theta)d\theta}{\int_\Theta L_n(X^n;\theta)\pi(\theta)d\theta} = \frac{\int_{\mathbb{U}_n} (\theta_0+u/n) L_n(X^n;\theta_0+u/n)\pi(\theta_0+u/n)du}
{\int_{\mathbb{U}_n}  L_n(X^n;\theta_0+u/n)\pi(\theta_0+u/n)du} =\\
 \theta_0 + \frac 1 n \frac{\int_{\mathbb{U}_n} u L_n(X^n;\theta_0+u/n)\pi(\theta_0+u/n)du}
{\int_{\mathbb{U}_n}  L_n(X^n;\theta_0+u/n)\pi(\theta_0+u/n)du}
\end{multline*}
and thus
$$
n\big(\widetilde \theta_n-\theta_0\big) =
\frac{\int_{\mathbb{U}_n} u Z_n(u)\pi(\theta_0+u/n)du}
{\int_{\mathbb{U}_n}  Z_n(u)\pi(\theta_0+u/n)du}.
$$
The right hand side is a functional of $Z_n(u)$, $u\in \mathbb{U}_n$, which under appropriate tightness conditions,
converges weakly to the random variable
$$
\tilde u = \frac{\int_{\Real} u Z(u)du}
{\int_{\Real}  Z(u)du},
$$
if the finite dimensional distributions of $Z_n(u)$ converge to those of $Z(u)$.
More precisely, the result claimed in Theorem \ref{thm} follows from Theorem I.10.2
in \cite{IH}, whose assumptions we check  in subsections \ref{sec-fdd} and \ref{sec-tight} below.

\subsection{Convergence of finite dimensional distributions}\label{sec-fdd}

We shall prove that the characteristic functions of the finite dimensional distributions of $\log$-likelihoods  $Y_n(u):=\ln Z_n(u)$
converge to those of the compound Poisson process  in \eqref{Y}. To this end, we will show that
for any $d\ge 1$ and real numbers $u_{-d}<...< u_d$
\begin{multline}\label{Poi}
\psi_n(\lambda) := \E_{\theta_0}\exp\left(
\sum_{k=-d}^d \iunit\lambda_k \Big(Y_n(u_k)-Y_n(u_{k-1})\Big)
\right)\xrightarrow{n\to\infty} \\
\prod_{k=-d}^d \exp\bigg(\varpi(u_k-u_{k-1})
\Big(
e^{-\frac \iunit 2 \lambda_k \beta^2-\frac 1 2 \beta^2\lambda_k^2 }-1
\Big)
\bigg)
=:\psi(\lambda), \quad \forall \lambda \in \Real^d,
\end{multline}
uniformly over compacts from $\Theta$,
where $\varpi$ and $\beta$ are constants, defined in Theorem \ref{thm}.
Without loss of generality, we shall assume that $u_0=0$ and consider only positive $u_k$'s. The symmetric case of negative
$u_k$'s is treated similarly and independence of the emerging compound Poisson processes $\Pi^+$ and $\Pi^-$ will be evident
from the proof.

Let $\widehat \Xi_j$ be the vector with the entries
\begin{equation}
\label{Xi}
\widehat \Xi^k_j:=\widehat \xi_j(\theta_0+u_k/n), \quad k=0,...,d.
\end{equation}
Using the expression \eqref{lh} for the likelihood, we get
\begin{equation}
\label{S}
\begin{aligned}
Y_n(u_k)-&Y_n(u_{k-1}) =\\
&  -\frac 1 2 \frac 1{1+\gamma}\sum_{j=1}^n \left(X_j-f(X_{j-1},\theta_0+u_k/n)-\widehat \Xi^k_{j-1}\right)^2 +\\
& \frac 1 2 \frac 1{1+\gamma}\sum_{j=1}^n \left(X_j-f(X_{j-1},\theta_0+u_{k-1}/n)-\widehat \Xi^{k-1}_{j-1}\right)^2 =\\
&  \sum_{j=1}^n -\frac 1 2 \frac 1{1+\gamma}
\left(
X_{j-1} (\rho^+-\rho^-) \one{X_{j-1}\in D^k_n}
-\widehat \Xi^k_{j-1}
+\widehat \Xi^{k-1}_{j-1}
\right)\times\\
&
\left(
2X_j -f(X_{j-1},\theta_0+u_{k}/n)-f(X_{j-1},\theta_0+u_{k-1}/n)
-\widehat \Xi^k_{j-1}
-\widehat \Xi^{k-1}_{j-1}
\right)
\end{aligned}
\end{equation}
where we set $D_n^k:= \Big[\theta_0 + \frac {u_{k-1}}{n}, \theta_0 + \frac{u_k}{n}\Big]$ and used the identity
$$
f(x,\theta_0 + u_{k-1}/n)-f(x,\theta_0 + u_{k}/n) = x (\rho^+-\rho^-)\one{x\in D^k_n}.
$$
If we define
$$
\delta^k_j:= \widehat \Xi^{k-1}_j-\widehat \Xi^{k}_j, \quad \text{and} \quad
\sigma^{k-1}_j:=\widehat \Xi^{0}_j\; -\widehat \Xi^{k-1}_j,
$$
the expression \eqref{S} takes the following form under $\P_{\theta_0}$
\begin{multline}
\label{SS}
 Y_n(u_k)-Y_n(u_{k-1}) =
\sum_{j=1}^n-\frac 1{1+\gamma}\left(
X_{j-1} (\rho^+-\rho^-) \one{X_{j-1}\in D^k_n}
+\delta^k_{j-1}
\right)\times \\
\bigg(
\sqrt{1+\gamma}\widehat \eps_j +
X_{j-1}(\rho^+-\rho^-)\one{X_{j-1}\in B^{k-1}_n}+\sigma^{k-1}_{j-1}
+\\
\frac 1 2 \Big(X_{j-1}(\rho^+-\rho^-)\one{X_{j-1}\in D^k_n}+\delta^k_{j-1}\Big)
\bigg),
\end{multline}
with $B^{k-1}_n:= \Big[\theta_0,\theta_0+\frac{u_{k-1}}{n}\Big]$.
The sequences $(\delta^k_j)$ and $(\sigma^{k-1}_j)$ satisfy the recursions
\begin{equation}\label{delta}
\begin{aligned}
\delta^k_j =&
b  \delta^k_{j-1} -
c(\rho^+-\rho^-)X_{j-1}\one{X_{j-1}\in D^k_n}, \quad j\ge 1 \\
\delta^k_0=&0,
\end{aligned}
\end{equation}
and
\begin{equation}\label{sigma}
\begin{aligned}
\sigma^{k-1}_j =&
b\sigma^{k-1}_{j-1} -
c(\rho^+-\rho^-)X_{j-1}\one{X_{j-1}\in B^{k-1}_n}, \quad j\ge 1 \\
\sigma^{k-1}_0=&0
\end{aligned}
\end{equation}
where we set  $b:=\frac a{1+\gamma} $, $c:=\frac{a \gamma}{1+\gamma}$.
In what follows, both representations  \eqref{S} and  \eqref{SS} will be useful.

To prove the convergence \eqref{Poi}, we shall partition the terms in the sum  \eqref{S} or \eqref{SS}
into $n^{1/2}$ consecutive blocks of size $n^{1/2}$ and discard $n^{1/4}$ first terms in each block.
As shown in the Lemma \ref{lem-1} below, discarding the total of $n^{1/4} \cdot n^{1/2}$  terms does not alter the limit of the sum
and, by Lemma \ref{lem-2}, the remaining blocks become approximately independent due to the fast mixing of the process $(X_j,\xi_j,\widehat \Xi_j)$.
Moreover, in each remained block, the probability of having exactly one of the events
$\{X_{j-1}\in D^k_n\}$ occurred is of order $n^{1/2}$. Hence the sum of $n^{1/2}$ such
nearly independent blocks yields the compound Poisson limit of Lemma \ref{lem-3}. This approach to Poisson limits
dates back  to at least \cite{M73}.

Denote  by $s_{j,k}$ the summands in  the right hand side of \eqref{S} or \eqref{SS}. Set
$$
S_n:=   \sum_{k=1}^d \lambda_k \sum_{j=1}^n s_{j,k},
$$
and, for $m=1,...,n^{1/2}$, define
$$
S_{n,m} =
\sum_{k=1}^d \lambda_k \sum_{j=(m-1)n^{1/2}+n^{1/4}}^{m n^{1/2}}
  s_{j,k}.
$$
For $\psi_n(\lambda)$ and $\psi(\lambda)$ defined in \eqref{Poi},
the triangle inequality yields the bound
\begin{equation}\label{tri}
\begin{aligned}
\big|\psi_n(\lambda)-\psi(\lambda)\big| \le
\E_{\theta_0} & \bigg|\exp\big(\iunit S_n\big) - \exp\bigg(\iunit \sum_{m=1}^{n^{1/2}}S_{n,m}\bigg)
\bigg| + \\
& \bigg|
\E_{\theta_0}\exp\bigg(\iunit \sum_{m=1}^{n^{1/2}}S_{n,m}\bigg)
-
\bigg(\widetilde \E_{\theta_0} \exp\big(\iunit  S_{n,1}\big)\bigg)^{n^{1/2}}
\bigg|+
\\
&
\bigg|
\bigg(\widetilde \E_{\theta_0} \exp\big(\iunit  S_{n,1}\big)\bigg)^{n^{1/2}}
-
\psi(\lambda)
\bigg|,
\end{aligned}
\end{equation}
where $\widetilde \E_{\theta_0}$ stands for the expectation with respect to the probability $\widetilde \P_{\theta_0}$ on $(\Omega,\F)$, under which
the process $(X_j, \xi_j, \widehat \Xi_j)$ is stationary (see Lemma \ref{lem-lip}). In the following lemmas we show that all three terms in the right hand side of \eqref{tri} vanish as $n\to\infty$,
uniformly over $\theta_0$ on compacts from $\Theta$.

\begin{lem}\label{lem-1}
For any $\lambda \in \Real^d$ and a compact $K\subset \Theta$,
\begin{equation}
\label{eq4.6}
\lim_n \sup_{\theta_0\in K}\E_{\theta_0}\bigg|S_n - \sum_{m=1}^{n^{1/2}}S_{n,m}\bigg|=0,
\end{equation}
and consequently
\begin{equation}
\label{eq4.7}
\lim_n \E_{\theta_0}  \bigg|\exp\big(\iunit S_n\big) - \exp\bigg(\iunit \sum_{m=1}^{n^{1/2}}S_{n,m}\bigg)
\bigg| =0,
\end{equation}
uniformly over $\theta_0\in K$.
\end{lem}
\begin{proof}
We shall assume that $n$ is large enough so that $\max_k |\theta_0+u_k/n|\le \sup|K|+1$
and hence on the events $\{X_{j-1}\in D^k_n\}$ and $\{X_{j-1}\in B^k_n\}$, we have $|X_{j-1}|\le \sup|K|+1$.
Using the representation \eqref{SS}, we get
\begin{align}\label{eq:3.18}
&\E_{\theta_0}\bigg|S_n - \sum_{m=1}^{n^{1/2}}S_{n,m}\bigg| \le \nonumber \\
& n^{3/4}
d \max_k |\lambda_k|\max_{j\le n}\E_{\theta_0}
\left|
X_{j-1} (\rho^+-\rho^-) \one{X_{j-1}\in D^k_n}
+\delta^k_{j-1}
\right|\times  \\
&
\nonumber
\bigg|
\sqrt{1+\gamma}\widehat \eps_j +
X_{j-1}(\rho^+-\rho^-)\one{X_{j-1}\in B^{k-1}_n}+\sigma^{k-1}_{j-1}
+
 \frac 1 2 \Big(X_{j-1}(\rho^+-\rho^-)\one{X_{j-1}\in D^k_n}+\delta^k_{j-1}\Big)
\bigg|
\end{align}
By Jensen's inequality, it follows from \eqref{delta}, that
$$
\E_{\theta_0}\big(\delta^k_j \big)^2\le
|b|  \E_{\theta_0}\big(\delta^k_{j-1}\big)^2 +\frac 1 {1-|b|}
c^2(\rho^+-\rho^-)^2(\sup|K|+1)^2\P_{\theta_0}(X_{j-1}\in D^k_n),
$$
which, in view of \eqref{probb} and $\delta^k_0=0$, implies $\max_{j\le n}\E_{\theta_0}\big(\delta^k_j \big)^2\le C_1/n$.
Further, since $\widehat \eps_j$ is independent
of $\F^X_{j-1}$,  $\E_{\theta_0}|\delta^k_{j-1}| |\widehat \eps_j|=\E_{\theta_0}|\delta^k_{j-1}| \E_{\theta_0} |\widehat \eps_j|\le C_2/n$.
Similarly
$\E_{\theta_0}\big(\sigma^{k-1}_j\big)^2\le C_3/ n$. Plugging these bounds into \eqref{eq:3.18}, we obtain
\eqref{eq4.6}
$$
\E_{\theta_0}\bigg|S_n - \sum_{m=1}^{n^{1/2}}S_{n,m}\bigg|\le C_4 n^{-1/4},
$$
with a constant $C_4$, depending only on $K$. The uniform convergence in \eqref{eq4.7} follows, since $|\exp(\iunit x)-\exp(\iunit y)|\le |x-y|$.
\end{proof}

\begin{lem}\label{lem-2}
For any $\lambda\in \Real^d$,
$$
\lim_n  \bigg|
\E_{\theta_0}\exp\bigg(\iunit \sum_{m=1}^{n^{1/2}}S_{n,m}\bigg)
-
\bigg(\widetilde \E_{\theta_0} \exp\big(\iunit  S_{n,1}\big)\bigg)^{n^{1/2}}
\bigg|=0,
$$
uniformly over $\theta_0\in K$ for any compact $K\subset\Theta$.
\end{lem}

\begin{proof}
We shall use the bound \eqref{B} of Lemma \ref{lem-lip} and thus will need to establish the corresponding Lipschitz property. To this end,
for fixed $x,y\in \Real$ and $z\in \Real^{d+1}$, let $\big(X_j(x), \xi_j(y), \widehat \Xi_j(z)\big)$ be the solution of the recursions
\eqref{ar}, \eqref{xi} and (cf. \eqref{KF1c} and \eqref{Xi})
$$
\widehat \Xi^k_j = b \widehat \Xi^k_{j-1} + c \Big(X_j-f(X_{j-1},\theta_0+u_k/n)\Big), \quad k=0,...,d,
$$
subject to the initial conditions $x$, $y$ and $z$ respectively. The latter recursions give
\begin{equation}
\label{expl}
\widehat \Xi^k_j(z) = b^{j}z_k + c\sum_{i = 1}^j \Big(X_i-f(X_{i-1},\theta_0+u_k/n)\Big)b^{j-i}.
\end{equation}
Consider the random variable (cf. the right hand side of \eqref{S})
\begin{align*}
\Phi_\ell(x,y,z):=&  \sum_{k=1}^d\lambda_k\sum_{j=1}^\ell-\frac 1 2 \frac 1{1+\gamma}
\left(
X_{j-1} (\rho^+-\rho^-) \one{X_{j-1}\in D^k_n}
-\widehat \Xi^k_{j-1}(z)
+\widehat \Xi^{k-1}_{j-1}(z)
\right)\times\\
&
\left(
2X_j -f(X_{j-1},\theta_0+u_{k}/n)-f(X_{j-1},\theta_0+u_{k-1}/n)
-\widehat \Xi^k_{j-1}(z)
-\widehat \Xi^{k-1}_{j-1}(z)
\right),
\end{align*}
where we dropped the dependence  on $x$ and $y$ for brevity. Define the function
$h(x,y,z):= \E_{\theta_0}\exp\big(\iunit \Phi_\ell(x,y,z)\big)$. We aim to show that for $z, z'\in \Real^{d+1}$,
\begin{equation}
\label{LipLip}
\big|h(x,y,z)-h(x,y,z')\big| \le L\Big(1+|x|+|y|+\|z\|+\|z'\|\Big)\|z-z'\|
\end{equation}
for some constant $L$, independent of $\ell$. Using the definition of $\Phi_\ell(x,y,z)$ and the explicit formula \eqref{expl}, a tedious but straightforward calculation gives
\begin{multline*}
\big|\Phi_\ell(x,y,z)-\Phi_\ell(x,y,z')\big|\le  \\
\sum_{k=1}^d |\lambda_k|\left(
|z_k-z'_k| +
|z_{k-1}-z'_{k-1}|\right)
 \Bigg(
\sum_{j=1}^\ell 2|b|^{j} \Big(
|X_j| + |X_{j-1}|
\Big)
+ \\
\big(\|z\|+\|z'\|\big)\frac 1{1-|b|}
 +
2
\sum_{j=1}^\ell  |b|^j\sum_{i = 1}^j
\Big(|X_i|+|X_{i-1}|+|\xi_{i-1}|\Big)|b|^{j-i}
\Bigg).
\end{multline*}
Taking the expectation of both sides, we get
\begin{multline*}
\big|h(x,y,z)-h(x,y,z')\big|\le \E_{\theta_0}\big|\Phi_\ell(x,y,z)-\Phi_\ell(x,y,z')\big|\le
\frac {2\|\lambda\|}{1-|b|}\|z-z'\|\times \\
 \Bigg(
\sum_{j=1}^\ell 2|b|^{j} 2
\E_{\theta_0}|X_j|
+ \big(\|z\|+\|z'\|\big)
 +
4
\sum_{j=1}^\ell  |b|^j\sum_{i = 1}^j |b|^{j-i}
\E_{\theta_0}\Big(|X_{i-1}|+|\xi_{i-1}|\Big)
\Bigg)
\end{multline*}
and applying the bound \eqref{momb}, the inequality \eqref{LipLip} follows.

Note that by the Markov property,
$
\widetilde \E_{\theta_0}
e^{\iunit S_{n,1}}
=\widetilde \E_{\theta_0} h(X_0,\xi_0,\widehat \Xi_0),
$
and for $m=1,...,n^{1/2}$
\begin{multline*}
\E_{\theta_0}\Big(
e^{\iunit S_{n,m}}\big|\F_{(m-1)n^{1/2}}
\Big) = \\
\E_{\theta_0}\bigg(
h\Big(
X_{(m-1)n^{1/2}+n^{1/4}-1},
\xi_{(m-1)n^{1/2}+n^{1/4}-1},
\widehat \Xi_{(m-1)n^{1/2}+n^{1/4}-1}\Big)
\big|\F_{(m-1)n^{1/2}}
\bigg).
\end{multline*}
Hence by the Lemma \ref{lem-lip}
$$
\bigg|\E_{\theta_0}\Big(
e^{\iunit S_{n,m}}\big|\F_{(m-1)n^{1/2}}
\Big)-
\widetilde \E_{\theta_0}
e^{\iunit S_{n,1}}
\bigg|\le
C_1  q^{n^{1/4}},
$$
with a positive constant $q<1$ and $C_1$, independent of $\theta_0$.
Finally, considering the telescopic series, we get
\begin{align*}
&\Bigg|\E_{\theta_0}\exp\bigg(\iunit \sum_{m=1}^{n^{1/2}}S_{n,m}\bigg)
-
\bigg(\widetilde \E_{\theta_0} \exp\big(\iunit  S_{n,1}\big)\bigg)^{n^{1/2}}
\Bigg|= \\
&\Bigg|\sum_{k=0}^{n^{1/2}-1}
\bigg(
\E_{\theta_0} \prod_{m=1}^{n^{1/2}-k}
e^{\iunit S_{n,m}}
\prod_{m=n^{1/2}-k+1}^{n^{1/2}}\widetilde \E_{\theta_0} e^{\iunit  S_{n,1}}
-
\E_{\theta_0} \prod_{m=1}^{n^{1/2}-k-1}
e^{\iunit S_{n,m}}
\prod_{m=n^{1/2}-k}^{n^{1/2}}\widetilde \E_{\theta_0} e^{\iunit  S_{n,1}}
\bigg)\Bigg|\le \\
&\sum_{k=0}^{n^{1/2}-1}
\E_{\theta_0}\bigg|
\E_{\theta_0}\Big(
e^{\iunit S_{n,n^{1/2}-k}}\big|\F_{(n^{1/2}-k-1)n^{1/2}}
\Big)
-
\widetilde \E_{\theta_0} e^{\iunit  S_{n,1}}
\bigg|\le C_1 n^{1/2} q^{n^{1/4}}\xrightarrow{n\to\infty}0,
\end{align*}
as claimed.

\end{proof}

\begin{lem}\label{lem-3}
For any $\lambda\in \Real^d$,
$$
\lim_n\bigg|
\bigg(\widetilde \E_{\theta_0} \exp\big(\iunit  S_{n,1}\big)\bigg)^{n^{1/2}}
-
\psi(\lambda)
\bigg|=0,
$$
uniformly over $\theta_0\in K$ for any compact $K\subset\Theta$.
\end{lem}

\begin{proof}
Let $D_n:= \bigcup_{k=1}^d D^k_n$ and define the  events
\begin{align*}
A_0 &=\bigcap_{j=1}^{n^{1/2}}\big\{X_{j-1}\not\in D_n\big\} & \text{(all samples avoid $D_n$)} \\
A_{j,k} & =  \big\{X_{j-1}\in D^k_n\big\}\cap \bigcap_{i\ne j,\; i=1}^{n^{1/2}}\big\{X_{i-1}\not\in D_n\big\}
 & \text{\ (only $(j-1)$-th sample  falls in $D^k_n$ )}  \\
A_1 & =\bigcup_{k=1}^d\bigcup_{j=1}^{n^{1/2}} A_{j,k} &  \text{(a single sample falls in $D_n$)} \\
A_{2+} & = (A_0\cup A_1)^c &  \text{(two or more samples fall in $D_n$)},
\end{align*}
and note that
\begin{equation}\label{plugme}
\widetilde \E_{\theta_0}e^{\iunit S_{n,1}} =
\widetilde \E_{\theta_0}e^{\iunit S_{n,1}}\one{A_0}+
\sum_{k=1}^d \sum_{j=1}^{n^{1/2}}\widetilde \E_{\theta_0}
e^{\iunit S_{n,1}}\one{A_{j,k}}
+
\widetilde \E_{\theta_0}e^{\iunit S_{n,1}}\one{A_{2+}}.
\end{equation}
Below we shall show that
\begin{align}
\label{Ajk}
 \sum_{j=1}^{n^{1/2}} \widetilde \E_{\theta_0}
e^{\iunit S_{n,1}}\one{A_{j,k}} & =
e^{-\frac \iunit 2 \lambda_k \beta^2-\frac 1 2 \beta^2\lambda_k^2 }(u_k-u_{k-1})\varpi n^{-1/2}  + O(n^{-3/4}) \\
 \label{A0}
\widetilde \E_{\theta_0}e^{\iunit S_{n,1}}\one{A_0} & =
1-\sum_{k=1}^d(u_k-u_{k-1})\varpi n^{-1/2}  + O(n^{-3/4}) \\
\label{A2+}
\widetilde \P_{\theta_0}(A_{2+}) & = O(n^{-3/4}),
\end{align}
where $\varpi:=\int p(\theta_0,y;\theta_0)dy$ and $p(x,y;\theta_0)$ is the unique invariant density
of the chain $(X_j,\xi_j)$ (see Lemma \ref{lem-erg}).
Plugging these expressions into \eqref{plugme}, we obtain
\begin{multline*}
\Big(\widetilde \E_{\theta_0}e^{\iunit S_{n,1}}\Big)^{n^{1/2}} = \\
\left(
1-\sum_{k=1}^d (u_k-u_{k-1})\varpi n^{-1/2}
+
\sum_{k=1}^d e^{-\frac \iunit 2 \lambda_k \beta^2-\frac 1 2 \beta^2\lambda_k^2 }(u_k-u_{k-1})\varpi n^{-1/2}  + O\big(n^{-3/4}\big)
\right)^{n^{1/2}}\\
\xrightarrow{n\to\infty}
\exp\left(
\sum_{k=1}^d\Big( e^{-\frac \iunit 2 \lambda_k \beta^2-\frac 1 2 \beta^2\lambda_k^2 }-1\Big)(u_k-u_{k-1})\varpi
\right).
\end{multline*}
The claimed result follows, once we check that \eqref{A0}, \eqref{Ajk} and \eqref{A2+} hold uniformly in $\theta_0$ on compacts from $\Theta$.

To this end, note that on the event $A_{j,k}$, the equations \eqref{delta} give
$$
\delta^\ell_i = \delta^\ell_0 b^i , \quad \ell\ne k
$$
and
$$
\delta^k_i = \delta^k_0 b^i-c(\rho^+-\rho^-)X_{j-1}\one{X_{j-1}\in D^k_n} b^{j-i}\one{i\ge  j},
$$
where $\delta^\ell_0$, $\ell =1,...,k$ are bounded random variables under $\widetilde\P_{\theta_0}$.
Similarly, since
$D^k_n\subset B^{\ell-1}_n$ for $\ell>k$, and  $\widetilde \P_{\theta_0}\big(X_{j-1}\in B^{\ell-1}_n\cap D^k_n\big)=0$ for $\ell\le k$,
$$
\sigma^{\ell-1}_i = \sigma^{\ell-1}_0 b^i, \quad \ell\in \{1,...,k\}
$$
and
$$
\delta^{\ell-1}_i = \sigma^{\ell-1}_0 b^i-c(\rho^+-\rho^-)X_{j-1}\one{X_{j-1}\in D^k_n} b^{j-i}\one{i\ge  j}, \quad \ell\in\{ k+1,...,d\}.
$$
Hence, for $n^{1/4}\le j\le n^{1/2}$ on the event $A_{j,k}$, by \eqref{SS} we have
\begin{align*}
S_{n,1}=&\sum_{\ell=1}^d \sum_{i=n^{1/4}}^{n^{1/2}} \lambda_\ell s_{i,\ell} =
\lambda_k \sum_{i=n^{1/4}}^{n^{1/2}}  s_{i,k} + O(b^{n^{1/4}})
 = \\
 & -\lambda_k\sum_{i=j}^{n^{1/2}}\frac 1{1+\gamma}\left(
X_{i-1} (\rho^+-\rho^-) \one{X_{i-1}\in D^k_n}
+\delta^k_{i-1}
\right)\times \\
& \bigg(
\sqrt{1+\gamma}\widehat \eps_i
+
 \frac 1 2 \Big(X_{i-1}(\rho^+-\rho^-)\one{X_{i-1}\in D^k_n}+\delta^k_{i-1}\Big)
\bigg) +O(b^{n^{1/4}})=\\
&
-
  \lambda_k \bigg(
  \frac {\rho^+-\rho^-}{\sqrt{1+\gamma}}
X_{j-1}\widehat \eps_j
+\frac 1 2
\Big(\frac {\rho^+-\rho^-}{\sqrt{1+\gamma}}\Big)^2
  X^2_{j-1}\bigg)
+
\\
&
\lambda_k\sum_{i=j+1}^{n^{1/2}}
\bigg(
\frac {c(\rho^+-\rho^-)}{\sqrt{1+\gamma}}X_{j-1} b^{j-i+1}
\widehat \eps_i
- \frac 1 2\left(\frac {c(\rho^+-\rho^-)}{\sqrt{1+\gamma}}X_{j-1} b^{j-i+1}\right)^2 \bigg)+O(b^{n^{1/4}})=\\
&\lambda_k \sum_{i=j}^{n^{1/2}}
\Big(
Q(j-i) X_{j-1}
\widehat \eps_i
- \frac 1 2 Q^2(i-j)X_{j-1}^2 \Big)+O(b^{n^{1/4}}):=\lambda_k \Psi_{j,k}+O(b^{n^{1/4}}),
\end{align*}
where we defined the kernel
$$
Q(i) = \begin{cases}
0 & i<0\\
-\dfrac {\rho^+-\rho^-}{\sqrt{1+\gamma}} & i= 0\\
\dfrac {c(\rho^+-\rho^-)}{\sqrt{1+\gamma}} b^{i-1} & i>0
\end{cases}
$$
By the triangle inequality
\begin{equation}
\label{trngl}
\begin{aligned}
\bigg|
\widetilde \E_{\theta_0}
e^{\iunit S_{n,1}} & \one{A_{j,k}}
-
e^{-\frac \iunit 2 \lambda_k \beta^2-\frac 1 2 \beta^2\lambda_k^2 }(u_k-u_{k-1})\varpi n^{-1}
\bigg|
\le\\
& \bigg|
\widetilde \E_{\theta_0}
e^{\iunit \lambda_k\Psi_{j,k}}\one{A_{j,k}} -
\widetilde \E_{\theta_0}
e^{\iunit \lambda_k\Psi_{j,k}}\one{X_{j-1}\in D^k_n}\bigg|
+ \\
& \bigg|
\widetilde \E_{\theta_0}
e^{\iunit \lambda_k\Psi_{j,k}}\one{X_{j-1}\in D^k_n}-
e^{-\frac \iunit 2 \lambda_k \beta^2-\frac 1 2 \beta^2\lambda_k^2 }(u_k-u_{k-1})\varpi n^{-1}
\bigg|+O(b^{n^{1/4}})
\end{aligned}
\end{equation}

By \eqref{sprobb}, for $i<j$
\begin{multline*}
\widetilde \P_{\theta_0}\big(X_i\in D_n, X_j\in D^k_n\big) = \widetilde \E_{\theta_0}
\one{X_i\in D_n } \int_{D^k_n} \frac 1 {\sqrt{2\pi}}e^{-\frac 1 2\big(x-f(X_{j-1},\theta_0)-\xi_{j-1}\big)^2}dx \le \\
\frac {u_k-u_{k-1}}{n}\widetilde \P_{\theta_0}(X_i\in D_n) \le C_1 n^{-2},
\end{multline*}
with  a constant $C_1$, independent of $\theta_0$.  Similar bound holds for $j<i$ and hence,
using the identity $\one{A}-\one{A\cap B}= \one{A\setminus (A\cap B)}=\one{A\cap B^c}$, we get
\begin{multline}\label{eq:3.21}
\bigg|
\widetilde \E_{\theta_0}
e^{\iunit \lambda_k\Psi_{j,k}}\one{A_{j,k}} -
\widetilde \E_{\theta_0}
e^{\iunit \lambda_k \Psi_{j,k}}\one{X_{j-1}\in D^k_n}\bigg| \le  \\
\widetilde \P_{\theta_0}\bigg(
\big\{X_{j-1}\in D^k_n\big\}\cap \bigcup_{i\ne j,\; i=n^{1/4}}^{n^{1/2}}\big\{X_{i-1}\in D_n\big\}
\bigg)\le
\\
 \sum_{i\ne j,\; i=n^{1/4}}^{n^{1/2}}\widetilde \P_{\theta_0}\big(X_{i-1}\in D_n, X_{j-1}\in D^k_n\big)\le
  C_1  n^{-3/2}.
\end{multline}
Further, since $(\widehat \eps_i)$ are i.i.d $N(0,1)$ and $\widehat \eps_i$ is independent of
$\F^X_{j-1}$ for $i\ge j$
\begin{align*}
&\widetilde \E_{\theta_0}
e^{\iunit \lambda_k\Psi_{j,k}}\one{X_{j-1}\in D^k_n}=\\
&
\widetilde \E_{\theta_0}\one{X_{j-1}\in D^k_n}
\exp\bigg(
-\frac 1 2 \lambda_k^2 X^2_{j-1} \sum_{i=j}^{n^{1/2}} Q^2(i-j)
- \iunit \frac 1 2 \lambda_k X_{j-1}^2 \sum_{i=j}^{n^{1/2}} Q^2(i-j)\bigg) =\\
&
\widetilde \E_{\theta_0}\one{X_{j-1}\in D^k_n}
\exp\bigg(
-\frac 1 2 \lambda_k^2 \theta_0^2 \sum_{i=0}^{n^{1/2}-j} Q^2(i)
- \iunit \frac 1 2 \lambda_k \theta_0^2  \sum_{i=0}^{n^{1/2}-j} Q^2(i)\bigg) + O(n^{-2})=\\
&
\widetilde \E_{\theta_0}\one{X_{j-1}\in D^k_n}
\exp\bigg(
-\frac 1 2  \lambda_k^2 \beta^2
- \iunit \frac 1 2 \lambda_k \beta^2
+\frac 1 2 \theta_0^2 \sum_{i=n^{1/2}-j+1}^{\infty} Q^2(i)(\lambda_k^2
+ \iunit  \lambda_k )
\bigg) + O(n^{-2}).
\end{align*}
By \eqref{sprobb},
\begin{multline*}
\widetilde \E_{\theta_0}\one{X_{j-1}\in D^k_n}
\exp\Big(
-\frac 1 2  \lambda_k^2 \beta^2
- \iunit \frac 1 2 \lambda_k \beta^2
\Big) = \\
\exp\Big(
-\frac 1 2  \lambda_k^2 \beta^2
- \iunit \frac 1 2 \lambda_k \beta^2
\Big) \varpi (u_k-u_{k-1})n^{-1} + O(n^{-2}),
\end{multline*}
and thus we have
\begin{equation}
\label{eq:3.31}
\begin{aligned}
&\bigg|
\widetilde \E_{\theta_0}
e^{\iunit \lambda_k\Psi_{j,k}}\one{X_{j-1}\in D^k_n}-
e^{
-\frac 1 2  \lambda_k^2 \beta^2
- \iunit \frac 1 2 \lambda_k \beta^2
} \varpi (u_k-u_{k-1})n^{-1}
\bigg|\le \\
&\bigg|
\widetilde \E_{\theta_0}
e^{\iunit \lambda_k\Psi_{j,k}}\one{X_{j-1}\in D^k_n}-
\widetilde \E_{\theta_0}
e^{
-\frac 1 2  \lambda_k^2 \beta^2
- \iunit \frac 1 2 \lambda_k \beta^2
}\one{X_{j-1}\in D^k_n}
\bigg|+O(n^{-2}) \le \\
&
\widetilde \E_{\theta_0}\one{X_{j-1}\in D^k_n}
\bigg|
\exp\bigg(
\frac 1 2 \theta_0^2 \sum_{i=n^{1/2}-j+1}^{\infty} Q^2(i)(\lambda_k^2
+ \iunit  \lambda_k )
\bigg)
-
1
\bigg|+O(n^{-2})\le \\
& \widetilde \P_{\theta_0}(X_{j-1}\in D^k_n)
C_2 b^{2(n^{1/2}-j)}+O(n^{-2}) \le
C_3 n^{-1} b^{n^{1/2}-j}+O(n^{-2}).
\end{aligned}
\end{equation}
Plugging \eqref{eq:3.21} and \eqref{eq:3.31} into \eqref{trngl}, we obtain
$$
\bigg|
\widetilde \E_{\theta_0}
e^{\iunit S_{n,1}}\one{A_{j,k}}
-
e^{-\frac \iunit 2 \lambda_k \beta^2-\frac 1 2 \beta^2\lambda_k^2 }(u_k-u_{k-1})\varpi n^{-1}
\bigg|
\le C_3 n^{-1} b^{n^{1/2}-j}+C_4 n^{-3/2},
$$
and in turn \eqref{Ajk}:
\begin{align*}
&
\left|\sum_{j=1}^{n^{1/2}} \widetilde \E_{\theta_0}
e^{\iunit S_{n,1}}\one{A_{j,k}}  -
e^{-\frac \iunit 2 \lambda_k \beta^2-\frac 1 2 \beta^2\lambda_k^2 }(u_k-u_{k-1})\varpi n^{-1/2}  \right| \le \\
&
\sum_{j=n^{1/4}}^{n^{1/2}}
\left| \widetilde \E_{\theta_0}
e^{\iunit S_{n,1}}\one{A_{j,k}}   -
 e^{-\frac \iunit 2 \lambda_k \beta^2-\frac 1 2 \beta^2\lambda_k^2 }(u_k-u_{k-1})\varpi n^{-1}
 \right|
 + \\
&
\left|
\sum_{j<n^{1/4}} \widetilde \E_{\theta_0}
e^{\iunit S_{n,1}}\one{A_{j,k}}\right|
+
\left|
\sum_{j<n^{1/4}} e^{-\frac \iunit 2 \lambda_k \beta^2-\frac 1 2 \beta^2\lambda_k^2 }(u_k-u_{k-1})\varpi n^{-1}
\right| \le \\
&  C_3 n^{-1}\sum_{j=n^{1/4}}^{n^{1/2}} b^{n^{1/2}-j}+C_4 n^{1/2} n^{-3/2} + n^{1/4}
\widetilde \P_{\theta_0}(A_{1,k}) + C_5 n^{-3/4} \le C_6 n^{-3/4}.
\end{align*}
By setting all $\lambda_k$'s to zero, we also get
\begin{equation}\label{AA1}
\widetilde \P_{\theta_0} (A_1) = \sum_{k=1}^d \sum_{j=1}^{n^{1/2}}
\widetilde \P_{\theta_0} (A_{j,k}) =
\sum_{k=1}^d(u_k-u_{k-1})\varpi n^{-1/2} +O(n^{-3/4}).
\end{equation}
Further,
\begin{multline}\label{AA0}
\widetilde \P_{\theta_0}(A_0) =
1- \widetilde\P_{\theta_0}\left(\bigcup_{j=1}^{n^{1/2}}\big\{X_{j-1}\in D_n\big\}\right) \ge \\
1- \sum_{k=1}^d \sum_{j=1}^{n^{1/2}}
\widetilde \P_{\theta_0}(X_{j-1}\in D^k_n) \stackrel{\dagger}{=}
1- \sum_{k=1}^d \sum_{j=1}^{n^{1/2}}
(u_k-u_{k-1})\varpi n^{-1} +O(n^{-3/2})= \\
1- \sum_{k=1}^d
(u_k-u_{k-1})\varpi n^{-1/2} +O(n^{-3/2})
\end{multline}
where in the equality $\dagger$ we used \eqref{sprobb}. On the other hand,
$\widetilde \P_{\theta_0} (A_0)\le 1-\widetilde \P_{\theta_0}(A_1)$ and the estimate  \eqref{A0}
follows from \eqref{AA1} and \eqref{AA0} and the asymptotic
$$
\widetilde \E_{\theta_0} e^{\iunit S_{n,1}}\one{A_0} = \widetilde \P_{\theta_0}(A_0) + O(b^{n^{1/4}}).
$$
Finally, \eqref{A2+} follows since
$\widetilde \P_{\theta_0}(A_{2+})=1- \widetilde \P_{\theta_0}(A_0)-\widetilde \P_{\theta_0}(A_1)=O(n^{-3/4})$.
\end{proof}

\subsection{Tightness}\label{sec-tight}
In this section we check the tightness conditions of  Theorem I.10.2, \cite{IH}.

\begin{lem}[condition 1.1 of Theorem I.10.2,  \cite{IH}]
For any compact $K\subset \Theta$, there is a constant $C>0$, such that
\begin{equation}
\label{holder}
\sup_{|u_1|\le R, |u_2|\le R}|u_2-u_1|^{-1}\E_{\theta_0} \left(\sqrt{Z_n(u_2)}-\sqrt{Z_n(u_1)}\right)^2 \le
C
\Big(1+R^2\Big),
\end{equation}
for all $\theta_0\in K$ and  $R\ge 0$.
\end{lem}

\begin{proof}
Suppose $u_2\ge u_1$, then using the elementary inequality $\ln \frac 1 x \ge  2(1-\sqrt x)$, $x>0$ we get
\begin{multline*}
\E_{\theta_0} \left(\sqrt{Z_n(u_2)}-\sqrt{Z_n(u_1)}\right)^2  = \E_{\theta_0+u_1/n }
\left(\sqrt{\frac{Z_n(u_2)}{Z_n(u_1)}}-1\right)^2 = \\
2\E_{\theta_0+u_1/n }
\left(1-\sqrt{\frac{Z_n(u_2)}{Z_n(u_1)}}\right) \le \E_{\theta_0+u_1/n }\ln \frac{Z_n(u_1)}{Z_n(u_2)}.
\end{multline*}
Similarly to \eqref{S}, we find that under $\P_{\theta_0+u_1/n }$,
\begin{multline*}
\ln \frac{Z_n(u_1)}{Z_n(u_2)}  =\frac 1{1+\gamma}
\sum_{j=1}^n\left(
X_{j-1} (\rho^+-\rho^-) \one{X_{j-1}\in D^2_n}
+\delta^2_{j-1}
\right)\times \\
\bigg(
\sqrt{1+\gamma}\widehat \eps_j
+\frac 1 2 \Big(X_{j-1}(\rho^+-\rho^-)\one{X_{j-1}\in D^2_n}+\delta^2_{j-1}\Big)
\bigg),
\end{multline*}
where $(\delta^2_j)$ is defined in \eqref{delta}. Note that
$\widehat \eps_j$ is independent
of $\F^X_{j-1}$ under $\P_{\theta_0+u_1/n}$ and, as in the  proof of Lemma \ref{lem-1},
$\E_{\theta_0+u_1/n}\big(\delta^2_j \big)^2\le C_1(u_2-u_1)n^{-1}$. Hence
\begin{multline*}
\E_{\theta_0} \left(\sqrt{Z_n(u_2)}-\sqrt{Z_n(u_1)}\right)^2 \le \\
\sum_{j=1}^n
\E_{\theta_0+u_1/n }\left(
X_{j-1} (\rho^+-\rho^-) \one{X_{j-1}\in D^2_n}
+\delta^2_{j-1}
\right)^2 \le C_2 |u_2-u_1|,
\end{multline*}
where $C_2$ depends only on $K$.
By symmetry, the same  inequality holds when $u_2\le u_1$ and \eqref{holder} follows.
\end{proof}

\begin{lem}[condition 1.2 of Theorem I.10.2 \cite{IH}]\label{lem-ld}
For any $p\ge 1$,  there is a constant $C(p)$, such that
$$
\E_{\theta_0}Z^{1/2}_n(u)\le \frac {C(p)}{|u|^p}, \quad u\in \mathbb{U}_n:=n(\Theta-\theta_0).
$$
\end{lem}

\begin{proof}
The proof is an adaptation  of the analogous Lemma 2.2 in \cite{ChK11}.
We shall assume $\theta_0>0$ and $u>0$, omitting the similar complementary cases (recall that $0\not\in \Theta$).
Note that for a constant $c>0$,
\begin{align*}
\E_{\theta_0}Z^{1/2}_n(u) =&\; \E_{\theta_0}Z^{1/2}_n(u)\one{Z^{1/2}_n(u)\ge e^{-c{u}}} +
\E_{\theta_0}Z^{1/2}_n(u)\one{Z^{1/2}_n(u)< e^{-c{u}}} \le \\
& \Big(\E_{\theta_0}Z_n(u)\Big)^{1/2}\Big(\P_{\theta_0}\big(Z^{1/2}_n(u)\ge e^{-c{u}}\big)\Big)^{1/2}+
e^{-\frac c 2 {u}} =\\
&\Big(\P_{\theta_0}\big(\ln Z^{1/2}_n(u)\ge -c {u} \big)\Big)^{1/2}+
e^{-\frac c 2 {u}},
\end{align*}
and hence it is enough to check the large deviation bound
\begin{equation}
\label{nado}
\P_{\theta_0}\Big( \frac 1 2\ln Z_n(u)\ge -c {u} \Big)\le \frac{C(p)}{{u}^{p}},
\end{equation}
for some positive constants $c$ and $C(p)$ and all $p\ge 1$.

For $u_1:=u>0$  the formula   \eqref{SS} gives
\begin{multline*}
\ln Z_n(u) =
-\frac 1{1+\gamma}
\sum_{j=1}^n\left(
X_{j-1} (\rho^+-\rho^-) \one{X_{j-1}\in D^1_n}
+\delta_{j-1}
\right)\times \\
\bigg(
\sqrt{1+\gamma}\widehat \eps_j
+ \frac 1 2 \Big(X_{j-1}(\rho^+-\rho^-)\one{X_{j-1}\in D^1_n}+\delta_{j-1}\Big)
\bigg)=
\sum_{j=1}^n \Big(\widehat \eps_jV_{j-1}  - \frac 1 2 V^2_{j-1}\Big)
\end{multline*}
where $D^1_n = [\theta_0,\theta_0+u/n]$, the sequence $(\delta_j)$ is generated by \eqref{delta} with $k=1$,  and
$$
V_{j-1} := -\frac 1 {\sqrt{1+\gamma}}\Big(X_{j-1}(\rho^+-\rho^-)\one{X_{j-1}\in D^1_n}+\delta_{j-1}\Big).
$$
Further,
\begin{align*}
& \P_{\theta_0}\Big(\frac 1 2 \ln Z_n(u)\ge -c {u} \Big) =\\
&
\P_{\theta_0}\bigg( \sum_{j=1}^n \bigg( \widehat \eps_j\Big(\frac 1 2 V_{j-1}\Big)
- \frac 1 2 \Big(\frac 1 2 V_{j-1}\Big)^2\bigg)
-\frac 1 8\sum_{j=1}^n  V^2_{j-1}
\ge -c {u} \bigg) \le \\
& \P_{\theta_0}\bigg( \sum_{j=1}^n \bigg( \widehat \eps_j\Big(\frac 1 2 V_{j-1}\Big)
- \frac 1 2 \Big(\frac 1 2 V_{j-1}\Big)^2\bigg)
\ge  c {u} \bigg)+
\P_{\theta_0}\bigg(
-\frac 1 8\sum_{j=1}^n  V^2_{j-1}
\ge - 2 c {u} \bigg),
\end{align*}
and since
\begin{multline*}
\P_{\theta_0}\bigg( \sum_{j=1}^n \bigg( \widehat \eps_j\Big(\frac 1 2 V_{j-1}\Big)
- \frac 1 2 \Big(\frac 1 2 V_{j-1}\Big)^2\bigg)
\ge  c {u} \bigg)\le \\
e^{-c{u}}
\E_{\theta_0}\exp\bigg( \sum_{j=1}^n \bigg( \widehat \eps_j\Big(\frac 1 2 V_{j-1}\Big)
- \frac 1 2 \Big(\frac 1 2 V_{j-1}\Big)^2\bigg)
 \bigg)=e^{-c{u}},
\end{multline*}
the bound \eqref{nado} holds, if we show that for some positive constant $c$
and all $p\ge 1$,
\begin{equation}
\label{Vbnd}
\P_{\theta_0}\bigg(
\sum_{j=0}^{n-1}  V^2_{j}
\le  c {u} \bigg)\le \frac {C(p)}{{u}^p}.
\end{equation}

To this end, we shall split the consideration to the cases ${u} \le n^s$ and ${u}> n^s$, where $s\in (0,1)$ is a constant to be chosen later on,
depending on $p$ in \eqref{Vbnd}.

\subsubsection*{Case ${u}\le n^s$}
In this case,  $D^1_n \subseteq \big[\theta_0, \theta_0+\frac 1{n^{1-s}}\big]$ and since
$$
\delta_j =
b  \delta_{j-1} -
c(\rho^+-\rho^-)X_{j-1}\one{X_{j-1}\in D^1_n},\quad j\ge 1,
$$
subject to $\delta_0=0$, it follows that
\begin{multline*}
|\delta_j| \le |c|  |\rho^+-\rho^-|\Big(\theta_0+ n^{s-1}\Big) \frac 1{1-|b|} = \\
\frac {|a|\gamma}{1 -|a| +\gamma} |\rho^+-\rho^-|\big(\theta_0+ n^{s-1}\big) \le
|a| |\rho^+-\rho^-|\big(\theta_0+ n^{s-1}\big),
\end{multline*}
where we used the definitions $c:=\frac{a\gamma}{1+\gamma}$ and $b:=\frac{a}{1+\gamma}$ and the assumption $|a|<1$. Further,
\begin{align*}
V^2_{j} =& \frac 1 {1+\gamma}\Big(X_{j}(\rho^+-\rho^-)\one{X_{j}\in D^1_n}+\delta_{j}\Big)^2\ge \\
& \frac 1 {1+\gamma}\Big(X_{j}(\rho^+-\rho^-)+\delta_{j}\Big)^2\one{X_{j}\in D^1_n} \ge \\
& \frac {(\rho^+-\rho^-)^2} {1+\gamma}\Big(\theta_0 -
|a| \big(\theta_0+ n^{s-1}\big)
\Big)^2\one{X_{j}\in D^1_n} \ge \\
&
 \frac {(\rho^+-\rho^-)^2} {1+\gamma}\frac 1 4 \theta_0^2 \big(1-|a|\big)^2\one{X_{j}\in D^1_n}:= C_1\one{X_{j}\in D^1_n}
\end{align*}
where the latter inequality holds for all $n$ large enough. Consequently,
$$
\P_{\theta_0}\bigg(
\sum_{j=0}^{n-1}  V^2_{j}
\le  c {u} \bigg) \le
\P_{\theta_0}\bigg(
\sum_{j=0}^{n-1}  \one{X_{j}\in D^1_n}
\le  \frac c{C_1} {u} \bigg)\le
\P_{\theta_0}\bigg(
\sum_{j=n^{1/2}}^{n-1}  \one{X_{j}\in D^1_n}
\le  \frac c{C_1} {u} \bigg)
$$
By Lemma \ref{lem-erg}, the process $(X_j)$ is  geometric mixing and we have
\begin{multline*}
\P_{\theta_0}\big(X_{j}\in D^1_n\big) = \E_{\theta_0} \int_{\theta_0}^{\theta_0+u/n} \frac{1}{\sqrt{2\pi}} e^{-\frac 1 2(t-\xi_{j-1}-f(X_{j-1},\theta_0))^2}dt\ge \\
\frac u n \frac{1}{\sqrt{2\pi}}\E_{\theta_0}\inf_{t\in [\theta_0, \theta_0+1]}
 e^{-\frac 1 2\big(t-\xi_{j-1}-f(X_{j-1},\theta_0)\big)^2}\ge  C_2\frac {{u}} n,
\end{multline*}
where the constant $C_2>0$ can be chosen for all $\theta_0\in K$ to be independent of $j$ and $n$ by the ergodic properties of $(X_j,\xi_j)$
from Lemma \ref{lem-erg}.
Hence with $c:=\frac 1 2 C_1C_2$, for an integer $p\ge 1$,
\begin{multline*}
\P_{\theta_0}\bigg(
\sum_{j=0}^{n-1}  V^2_{j}
\le  c {u} \bigg)\le \P_{\theta_0}\bigg(
\sum_{j=n^{1/2}}^{n-1}  \one{X_{j}\in D^1_n}
\le  \frac c{C_1} {u} \bigg)\le \\
\P_{\theta_0}\bigg(
\bigg|\sum_{j=n^{1/2}}^{n-1}  \big(\one{X_{j}\in D^1_n}-\E_{\theta_0}\one{X_{j}\in D^1_n}\big)\bigg|
\ge\frac 1 3  C_2{u}\bigg)\le\\
(C_2/3)^{-2p}\frac 1 {{u}^{2p}}\E_{\theta_0}\left(\sum_{j=n^{1/2}}^{n-1} \eta_j\right)^{2p},
\end{multline*}
where we defined $\eta_j:=\one{X_{j}\in D^1_n}-\E_{\theta_0}\one{X_{j}\in D^1_n}$.
Since $|\eta_j|\le 2$, by Lemma \ref{lem-erg},
\begin{multline*}
 \bigg|\E_{\theta_0}\bigg(\sum_{j=n^{1/2}}^{n-1} \eta_j\bigg)^{2p}-
\widetilde \E_{\theta_0}\bigg(\sum_{j=n^{1/2}}^{n-1} \eta_j\bigg)^{2p}
\bigg|\le \\
 n^{2p}\E_{\theta_0}\bigg|\E_{\theta_0}\bigg(\bigg(\frac 1 n \sum_{j=n^{1/2}}^{n-1} \eta_j\bigg)^{2p}\bigg|\F_{0}\bigg)-
\widetilde \E_{\theta_0}\bigg(\frac  1 n \sum_{j=n^{1/2}}^{n-1} \eta_j\bigg)^{2p}
\bigg|\le C n^{2p} r^{n^{1/2}} \le 1,
\end{multline*}
for all $n$ large enough.
Hence it is enough to check
\begin{equation}
\label{stat}
\widetilde \E_{\theta_0}\left(\sum_{j=n^{1/2}}^{n-1} \eta_j\right)^{2p}\le C(p) {u}^p.
\end{equation}
To estimate the latter expectation, we shall apply the covariance inequality (8.1) from  \cite{DD03}:
\begin{multline}
\label{DDineq}
\widetilde \E_{\theta_0}\left(\sum_{j=n^{1/2}}^{n-1}  \eta_j \right)^{2p}
=
\widetilde \E_{\theta_0}\left(\sum_{j=0}^{n-n^{1/2}-1} \eta_j\right)^{2p}\le  \\
\left(4pn\sum_{j=0}^{n}\Big(\widetilde\E_{\theta_0}\big|\eta_0 \widetilde\E_{\theta_0}(\eta_j|\F_0)\big|^p\Big)^{1/p} \right)^{p}.
\end{multline}
Since $|\eta_0|\le 2$,
\begin{equation}
\label{eq:2.27}
\widetilde\E_{\theta_0}\big|\eta_0\widetilde \E_{\theta_0}(\eta_j|\F_0)\big|^p\le
4^{p}\; \widetilde\E_{\theta_0}\big|\widetilde \E_{\theta_0}(\eta_j|\F_0)\big|^p
\le
4^{p}\; \Big(\frac u n C_4 C r^j \Big)^p,
\end{equation}
where the latter inequality holds by Lemma \ref{lem-erg}, since
$\widetilde \E_{\theta_0}(\eta_j|\F_0)=\widetilde\E_{\theta_0}\big(\widetilde\E_{\theta_0}(\eta_j|\F_{j-1})\big|\F_0\big)$
and $\big|\widetilde\E_{\theta_0}(\eta_j|\F_{j-1})\big|\le  C_4\frac u n $.
Plugging the bound \eqref{eq:2.27} into \eqref{DDineq}, we obtain \eqref{stat} and consequently \eqref{Vbnd} for  ${u}\le n^s$.

\subsubsection*{Case ${u}> n^s$}
For a fixed integer $k$ and all $j> k$, define  $\Gamma_{j-1,k}:=\bigcap_{i=j-k}^{j-1} \{X_i\not\in\Theta\}$ and note that on $\Gamma_{j-1,k}$ we have
$$
|\delta_j|/|\rho^+-\rho^-| =\bigg|
c \sum_{i=1}^{j-k} b^{j-i} X_{i-1}\one{X_{i-1}\in D^1_n}\bigg|\le
|c| \sup |\Theta| \frac{|b|^k}{1-|b|}=: C_1 b^k.
$$
Now let $k'$ be such that, $C_1 b^{k'} \le \frac 1 2 \theta_0$, then (recall that both $\theta_0$ and $u$ are positive)
\begin{multline*}
V^2_{j} \ge
\frac 1 {1+\gamma}\Big(X_{j}(\rho^+-\rho^-)+\delta_{j}\Big)^2 \one{X_{j}\in D^1_n}\one{\Gamma_{j-1,k'}}\ge \\
\frac {(\rho^+-\rho^-)^2} {1+\gamma}\frac 1 4 \theta_0^2 \one{X_{j}\in D^1_n}\one{\Gamma_{j-1,k'}}=:
C_2 \one{X_{j}\in D^1_n}\one{\Gamma_{j-1,k'}}.
\end{multline*}
Define $W_j:= \one{X_{j}\in D^1_n}\one{\Gamma_{j-1,k'}}$ and note that
\begin{multline*}
\E_{\theta_0} W_j
 =\E_{\theta_0}\one{\Gamma_{j-1,k}} \P_{\theta_0}\big(X_{j}\in D^1_n|X_{j-1}, \xi_{j-1}\big)  \ge  \\
\frac u n
 \E_{\theta_0}\one{\Gamma_{j-1,k'}}
\frac 1 {\sqrt{2\pi}}\inf_{t\in \Theta}e^{-\frac 1 2\big(t-\xi_{j-1}-f(X_{j-1};\theta_0)\big)^2} \ge C_3 \frac {{u}} n,
\end{multline*}
where the positive constant $C_3$ can be chosen independent of $j$ and $n$ due to the ergodic properties of $(X_j,\xi_j)$ from Lemma \ref{lem-erg}.
Hence with $c:=\frac 1 2 C_3C_2$ and any integer $p>1$,
\begin{multline*}
\P_{\theta_0}\bigg(
\sum_{j=0}^{n-1}  V^2_{j}
\le  c {u} \bigg)\le
\P_{\theta_0}\bigg(
\sum_{j=k'}^{n-1}  W_{j}
\le  \frac 1 2 C_3 {u} \bigg)
\le \\
\P_{\theta_0}\bigg(
\Big|\sum_{j=k'}^{n-1}  \big(W_{j}-\E_{\theta_0}W_j\big)\Big|
\ge  \frac 1 3 {u}C_3 \bigg)\le \left( 3/C_3\right)^{2p+2} \frac 1 {{u}^{2p+2}}
\E_{\theta_0}
\bigg(
\sum_{j=k'}^{n-1}  \big(W_{j}-\E_{\theta_0}W_j\big)\bigg)^{2p+2}
\end{multline*}
for all sufficiently large $n$.
Now \eqref{Vbnd} follows if we show that for a positive constant $C(p)$,
\begin{equation}
\label{showme}
\E_{\theta_0} \bigg(\sum_{j=k'}^{n-1}  \big(W_{j}-\E_{\theta_0} W_j\big)\bigg)^{2p+2}\le C(p) {u}^{p+2}, \quad\forall\; {u}>n^s.
\end{equation}
To this end, we shall use the Marcinkiewicz-Zygmund inequality from \cite{DL99}. For a sequence of random variables $(\eta_j)_{j\in \mathbb{N}}$,
the coefficient of weak dependence is defined
$$
C_{t,q} := \sup\big|\cov\big(\eta_{t_1}...\eta_{ t_m},\eta_{t_{m+1}}... \eta_{t_q}\big)\big|,
$$
where the supremum is taken over all $\{t_1,...,t_q\}$, such that $1\le t_1\le ...\le t_q$ and $m$, $t$
satisfy $t_{m+1}-t_m=t$.

\begin{thm}[Theorem 1, \cite{DL99}]
Let $(\eta_j)_{j\in \mathbb{N}}$ be a sequence of central random variables such that for a fixed integer $q\ge 2$,
\begin{equation}
\label{weakd}
C_{t,q}=O(t^{-q/2})\quad \text{as}\quad t\to\infty.
\end{equation}
Then there exists a positive constant $B$, independent of $n$, for which
\begin{equation}
\label{DLres}
\bigg|\E \bigg(\sum_{j=0}^{n-1} \eta_j\bigg)^q\bigg|\le B n^{q/2}.
\end{equation}
\end{thm}

We shall apply this theorem to the bounded sequence $\eta_j:=W_j-\E_{\theta_0}W_j$. Since $\eta_j$ is a function of $(X_{j-k'},...,X_j)$,
it inherits the mixing property \eqref{mix}. More precisely,  with
\begin{align*}
&h(x,y):=\E_{\theta_0}\big(\eta_{t_1}...\eta_{t_m}|X_{t_m}=x, \xi_{t_m}=y\big),\\
& g(x,y):=\E_{\theta_0}\big(\eta_{t_{m+1}}...\eta_{t_q}|X_{t_{m+1}-k'}=x, \xi_{t_{m+1}-k'}=y\big)
\end{align*}
by the Markov property of $(X_j,\xi_j)$,  for $t\ge k'$
\begin{multline*}
C_{t,q}=\Big|\E_{\theta_0}\eta_{t_1}...\eta_{t_q}-\E_{\theta_0}\eta_{t_1}...\eta_{t_m}\E_{\theta_0}\eta_{t_{m+1}}...\eta_{t_q}\Big|= \\
 \Big|\E_{\theta_0}h(X_{t_m},\xi_{t_m})g(X_{t_{m+1}-k'},\xi_{t_{m+1}-k'})-
\E_{\theta_0}h(X_{t_m},\xi_{t_m})\E_{\theta_0}g(X_{t_{m+1}-k'},\xi_{t_{m+1}-k'})\Big| \le \\
2^q C r^{t-k'}(1+r^{t_m}) \le 2^{q+1} C r^{t-k'},
\end{multline*}
which clearly satisfies \eqref{weakd}. Since $n< {u}^{1/s}$,   \eqref{DLres} now gives
$$
\E_{\theta_0} \bigg(\sum_{j=1}^{n}  \big(W_{j}-\E_{\theta_0} W_j\big)\bigg)^{2p+2}\le B n^{p+1}\le B {u}^{(p+1)/s},
$$
and the bound \eqref{showme} follows, if we choose $s:=(p+1)/(p+2)$.

\end{proof}

\section{Simulated experiments}\label{sec-sim}

The objective of this section is to illustrate the results of Theorem \ref{thm} by means of a simulation.
To this end, we fixed the following values of the parameters
$$
\theta_0=1.5, \ \rho^+=0.9,\ \rho_-=-0.5,\ a=0.9
$$
and estimated the root mean square errors of the Bayes estimator $\widetilde \theta_n$ and the ML estimator $\widehat \theta_n$
by averaging over a large number of Monte Carlo trials. This has been done in two ways: by computing the estimators, based on simulated data,
and computing the corresponding limit quantities, based on simulated  process from the limit experiment. The practical advantage, offered by
Theorem \ref{thm}, is that the latter simulation requires  much less CPU time than the former.

\begin{figure}
  \includegraphics[scale=0.6]{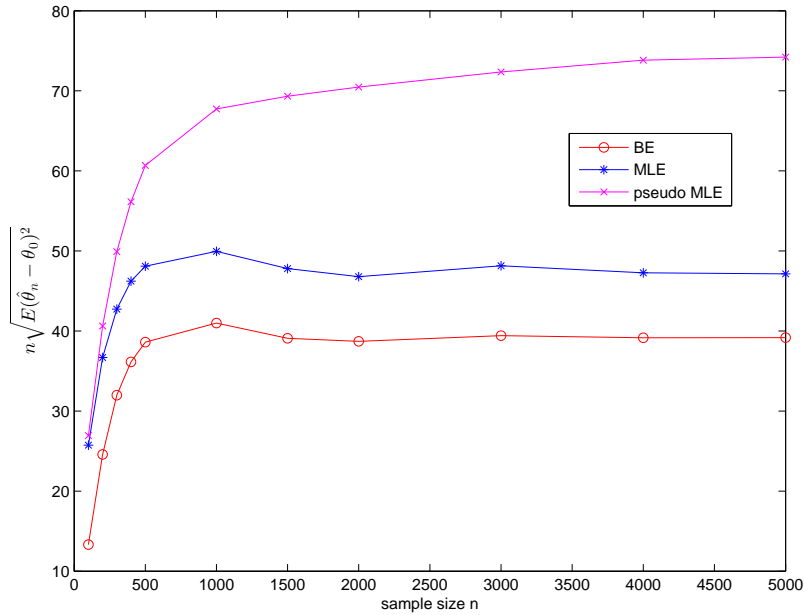}\\
  \caption{\label{fig1} The normalized empirical root mean squares of the BE, MLE and pseudo MLE versus the sample sizes}
\end{figure}

\subsection{Simulated data}

Using the recursions \eqref{ar} and \eqref{xi}, we generated a large number ($M=20.000$) of sample paths.
For each path we computed the Bayes estimator $\widetilde \theta_n$, using the formula \eqref{BE} and the
uniform prior on the interval $\Theta := (1,2)$. Then we calculated the normalized empirical root mean square error for a number of
sample sizes
$$
n \sqrt{\widehat\E_{\theta_0}\big(\widetilde \theta_n-\theta_0 \big)^2},
$$
where $\widehat\E_{\theta_0}$ denotes averaging over the paths.

Similarly, we computed the (central
\footnote{
Since the likelihood function \eqref{lh} is piecewise constant, the MLE is not unique.
}
) ML estimator and the {\em pseudo} ML estimator, which assumes independent innovations with the same variance $1+ 1/(1-a^2)$.
The results, depicted at Figure \ref{fig1}, indicate that the errors converge as the sample size $n$ increases
and that the Bayes estimator performs better than the others for smaller sample sizes as well.

\subsection{Simulated limit experiment}

While the distribution of the random variable $\tilde u$, defined in Theorem \ref{thm}, cannot be computed
in a closed form, it is easy to sample and the  expectations can be approximated by averaging over Monte Carlo trials.
To this end, we estimated the value of
$$
\varpi = \int_\Real p(\theta_0,y;\theta_0)dy \approx 0.0576
$$
using standard kernel estimator, applied to a single long trajectory of $(X_j)$. Figure \ref{fig2} depicts the estimated
marginal density
$$
\int_\Real p(x,y;\theta_0)dy, \quad x\in [-8,8].
$$
Next we generated a large number ($M=10^6$) of samples from the compound Poisson process $Z(u)$, defined in \eqref{Y} and computed
the approximate root mean square errors
$$
\sqrt{\widehat \E_{\theta_0}(\tilde u)^2}\approx 38.64 \quad \text{and}\quad \sqrt{\widehat \E_{\theta_0}(\hat u)^2}\approx 46.88,
$$
where $\hat u$ is the central maximizer of $Z(u)$ and $\widehat \E_{\theta_0}$ denotes the empirical expectation.
Note that the obtained estimates are at good correspondence with the plots at Figure \ref{fig1}.

The typical realization  of $Z(u)$ along with  $\tilde u$ and $\hat u$  are plotted at Figure \ref{fig3}.
The densities of $\hat u$ and $\tilde u$, whose kernel estimates are depicted at Figure \ref{fig4}, appear to be heavy tailed.

\begin{figure}
  \includegraphics[scale=0.6]{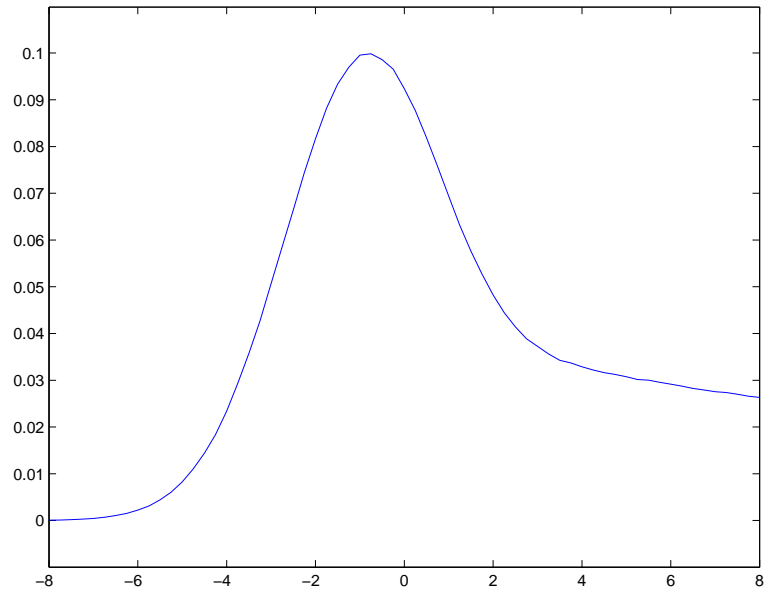}\\
  \caption{\label{fig2} The estimated stationary density of $(X_j)$
  }
\end{figure}

\begin{figure}
  \includegraphics[scale=0.6]{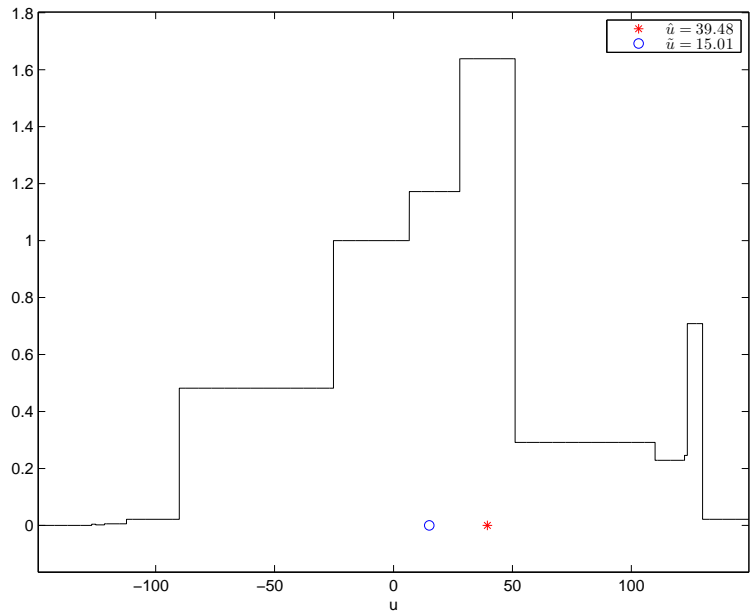}\\
  \caption{\label{fig3} A typical sample path of the limit likelihood ratio process $Z(u)$, $u\in \Real$. The marks are the positions of $\hat u$ and $\tilde u$ }
\end{figure}

\begin{figure}
  \includegraphics[scale=0.6]{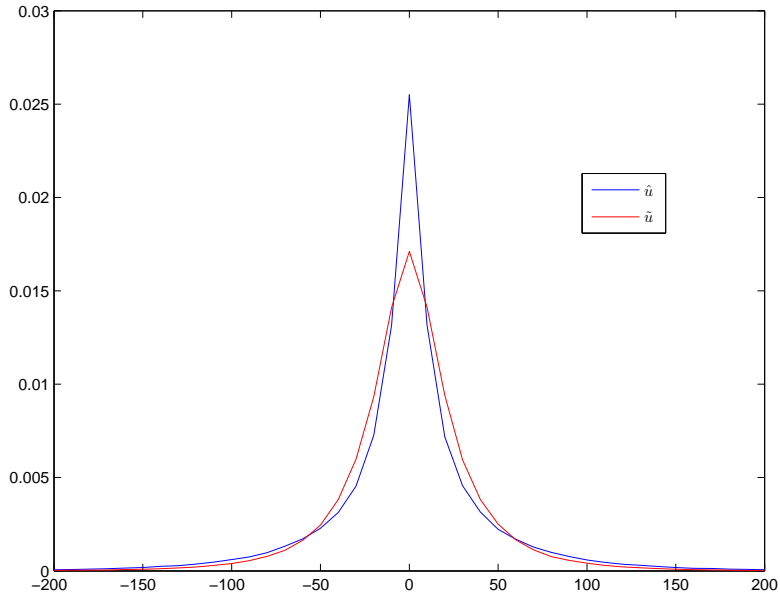}\\
  \caption{\label{fig4} The estimated  probability densities of $\tilde u$ and $\hat u$
  }
\end{figure}

\appendix
\section{Ergodic lemmas used in the proof}\label{app-sec}

The proofs in Section \ref{sec-pf} use the ergodic properties of the processes, summarized in the
following lemmas. Our standing assumption is \eqref{AE}.

\begin{lem}\label{lem-bnd}
For  all integers  $j\ge 0$ and $p\ge 1$,
\begin{equation}
\label{probb}
\P_{\theta_0}\big(X_j \in [\theta_0,\theta_0+v/n]\big) \le \frac{|v|}{n},
\end{equation}
and
\begin{equation}
\label{momb}
\E_{\theta_0} \Big(|X_j|^p + |\xi_j|^p\big|X_0=x, \xi_0=y\Big)\le  r_1^j R_1 \big(|x|^p+|y|^p\big) + R_2,
\end{equation}
with  a positive constant $r_1<1$ and constants $R_1$ and $R_2$, independent of $\theta_0$.
\end{lem}

\begin{proof}
For $j\ge 1$,
$$
\P_{\theta_0}\big(X_j \in [\theta_0,\theta_0+v/n]\big) =\E_{\theta_0} \int_{\theta_0}^{\theta_0+v/n}
\frac 1 {\sqrt{2\pi}} e^{-\frac 1 2 (x-f(X_{j-1},\theta_0)-\xi_{j-1})^2}dx \le  \frac {v} n.
$$
Further, by Jensen's inequality
$$
|\xi_j|^p \le \left(|a|\xi_{j-1}+(1-|a|) \frac 1 {1-|a|}|\zeta_j|\right)^p \le |a| |\xi_{j-1}|^p
+ (1-|a|)^{1-p} |\zeta_j|^p,
$$
and hence
$$
\E_{\theta_0}\big(|\xi_j|^p \big|\xi_0=y\big)\le |a|^j |y|^p + C_1.
$$
Similarly, with $\rho:= |\rho^+|\vee |\rho^-|$
$$
|X_j|^p \le \Big(\rho |X_{j-1}| + |\xi_{j-1}|+|\eps_j|\Big)^p= \\
\rho |X_{j-1}|^p +
\left(\frac {2} {1-\rho}\right)^{p-1}\Big(|\xi_{j-1}|^p+|\eps_j|^p\Big),
$$
and
$$
\E_{\theta_0}\Big(|X_{j}|^p\big|X_0=x,\xi_0=y\Big)\le \rho^j |x|^p + C_2 |a|^j |y|^p + C_3,
$$
which gives \eqref{momb}.
\end{proof}

\begin{lem}\label{lem-erg}
The Markov chain $(X_j, \xi_j)$  has the unique invariant measure under $\P_{\theta_0}$,
with uniformly bounded probability density $p(x,y;\theta_0)$ satisfying
\begin{multline}
\label{sprobb}
\widetilde \P_{\theta_0}\big(X_j \in [\theta_0,\theta_0+v/n]\big)= \\
\int_{\theta_0}^{\theta_0+v/n}\int_{\Real} p(x,y;\theta_0)dxdy= \frac{v}{n} \int_\Real p(\theta_0,y;\theta_0)dy+ O(n^{-2}),
\end{multline}
where $\widetilde \P_{\theta_0}$ is the corresponding  stationary probability on $(\Omega,\F)$.

Moreover, the chain is geometrically ergodic, i.e.
there exist positive constants $C$ and  $r<1$, such that for a measurable function $|h|\le 1$ and $m\ge k$
\begin{equation}
\label{ge1}
\Big|\E_{\theta_0}\big(h(X_m,\xi_m)|X_k=x,\xi_k=y\big)-\widetilde \E_{\theta_0}h(X_k,\xi_k)\Big|\le C r^{m-k} (|x|+|y|), \quad x,y\in \Real,
\end{equation}
and consequently,  for an $\F_{m,\infty}$-measurable random variable $|H|\le 1$
\begin{equation}
\label{ge}
\E_{\theta_0}\Big|\E_{\theta_0}\big(H|\F_k\big)-\widetilde \E_{\theta_0} H\Big|\le C r^{m-k}.
\end{equation}
Finally, $(X_j,\xi_j)$ is geometrically mixing, i.e. for measurable functions $|g|\le 1$, $|h|\le 1$
\begin{equation}
\label{mix}
\Big|\E_{\theta_0}g(X_k,\xi_k)h(X_{k+m},\xi_{k+m})-
\E_{\theta_0}g(X_k,\xi_k)\E_{\theta_0} h(X_{k+m},\xi_{k+m}) \Big|\le C r^m(1+r^k).
\end{equation}
In particular, \eqref{ge} and  \eqref{mix} hold with the stationary expectation $\widetilde \E_{\theta_0}$.
\end{lem}

\begin{proof}
The transition kernel of the process  $(X_j,\xi_j)$ has a positive density with respect to the
Lebesgue measure:
$$
(P\one{A})(x,y) := \int_A \frac 1 {2\pi} \exp\left(-\frac 1 2 \big(u-f(x,\theta_0)-y\big)^2-\frac 1 2 (v-ay)^2\right)dudv
$$
and hence in the terminology of \cite{MT09}, it is $\psi$-irreducible and aperiodic. Further, a ball $B_R$ of radius $R>0$ around the
origin is a small set with respect to e.g. the measure
$$
\nu(dx,dy):=e^{ -\frac 1 2 (u^2+ v^2)-(1+R)(|u| +|v|) }dudv
$$
and $V(x,y)=|x|+|y|$ satisfies the drift condition
$$
(PV)(x,y) -V(x,y)\le -\frac  1 2 (1-\rho \wedge |a|) V(x,y)+ 2\one{(x,y)\in B_R}, \quad (x,y)\in \Real^2,
$$
for sufficiently large $R$. By Theorem 15.0.1 \cite{MT09}, it follows that there exists a unique invariant probability measure
$\pi$ and for any measurable  $h(x,y)\le V(x,y)$,
$$
\Big|P^n h - \int hd\pi\Big|\le Cr^n V(x,y),
$$
with positive constants $C$ and $r<1$, i.e.   \eqref{ge1} holds. Since $\widetilde \E_{\theta_0} H = \widetilde \E_{\theta_0} h(X_m,\xi_m)$
and $\E_{\theta_0} (H|\F_k) = \E_{\theta_0} (h(X_m,\xi_m)|\F_k)$ with $h(x,y):= \E_{\theta_0}(H|X_m=x,\xi_m=y)$,
the claim \eqref{ge} follows from \eqref{ge1} and \eqref{momb}.
Since the transition kernel $P$ has a bounded continuously differentiable density with respect to the Lebesgue measure,
so does the invariant measure $\pi$ and \eqref{sprobb} follows.
The mixing inequality \eqref{mix} follows from
Theorem 16.1.5 in \cite{MT09}.
\end{proof}

The theory, used in the proof of the previous lemma, does not directly apply to the Markov chain
$(X_j,\xi_j,\widehat\Xi_j)$ (see \eqref{Xi} for the definition of  $\widehat \Xi_j$),
since  it is generated by a $(3+d)$-dimensional recursion, driven by two dimensional noise. This typically excludes
$\psi$-irreducibility. Fortunately, for our purposes the following weaker properties
are sufficient:

\begin{lem}\label{lem-lip}
The Markov process $(X_j,\xi_j,\widehat\Xi_j)$ has the unique invariant measure. Let $\widetilde \P_{\theta_0}$ denote the
corresponding stationary probability\footnote{by uniqueness, the stationary probabilities $\widetilde \P_{\theta_0}$, introduced in Lemmas
\ref{lem-erg} and \ref{lem-lip}, coincide}.
Then for a measurable function $h(x,y,z)$, satisfying $|h(x,y,z)|<1$ and the  Lipschitz condition
$$
|h(x,y,z)-h(x,y,z')|\le L\big(1+|x|+|y|+\|z\|+\|z'\|\big) \|z-z'\|, \quad x,y\in \Real,\; z,z'\in \Real^{d+1},
$$
with a positive constant  $L$,
\begin{equation}
\label{B}
\E_{\theta_0}\Big|\E_{\theta_0} \big(h(X_m,\xi_m, \widehat\Xi_m)\big|\F_\ell\big)-\widetilde \E_{\theta_0} h(X_0,\xi_0, \widehat\Xi_0)\Big|\le
C  q^{m-\ell}
\end{equation}
for some positive constants $C$ and $q<1$ and all integers $m\ge \ell\ge 0$.
\end{lem}

\begin{proof}
Under the stationary measure $\widetilde \P_{\theta_0}$ from Lemma \ref{lem-erg}, we can extend the definition of
$(X_j,\xi_j)$ to the negative integers and define
\begin{equation}
\label{Xi0}
\widehat \Xi^{k}_0 := c\sum_{i=-\infty}^0 b^i \big(X_i-f(X_{i-1},\theta_0+u_k/n)\big), \quad k = 0,...,d
\end{equation}
where  $b:=\frac{a}{1+\gamma}$ and $c:=\frac{a\gamma}{1+\gamma}$. The distribution of $(X_0,\xi_0,\widehat \Xi_0)$ is invariant.
To establish uniqueness, let $\mu$ and $\mu'$ be two invariant measures and note that by Lemma \ref{lem-erg} their $(X,\xi)$ marginals coincide.
Hence
$$
\mu(dx,dy,dz) = \nu(dx,dy)\mu(x,y;dz),\quad \mu'(dx,dy,dz) = \nu(dx,dy)\mu'(x,y;dz)
$$
where $\nu$ is the invariant measure of the process $(X_j,\xi_j)$ and $\mu(x,y;dz)$ and $\mu'(x,y;dz)$ are corresponding regular
conditional probabilities.  Let $(X_j,\xi_j,\widehat \Xi_j)$ and
$(X_j,\xi_j,\widehat \Xi'_j)$ be the solutions of the recursions \eqref{ar}, \eqref{xi} and  \eqref{KF1c} with $u:=u_k$, $k=0,...,d$
subject to the initial conditions $(X_0,\xi_0,\widehat\Xi_0)$ and $(X_0,\xi_0,\widehat\Xi'_0)$, where $(X_0,\xi_0)$ is sampled from $\nu$ and $\widehat \Xi_0$ and $\widehat \Xi'_0$
are sampled from $\mu(X_0,\xi_0;dz)$ and $\mu'(X_0,\xi_0;dz)$.
Note that $\widehat \Xi'_j-\widehat \Xi_j= b^j(\widehat \Xi'_0-\widehat \Xi_0)$ and
hence for any uniformly continuous function $g$
$$
\left|\int g d\mu -\int gd\mu'\right|\le \E_{\theta_0}\big|g(X_j,\xi_j,\widehat\Xi_j)-g(X_j,\xi_j,\widehat\Xi'_j)\big|
\xrightarrow{j\to\infty}0.
$$
Since uniformly continuous functions form a measure defining class, the uniqueness follows.

To derive the bound \eqref{B}, note that:
\begin{multline*}
\widehat \Xi^k_m = \widehat \Xi^k_\ell b^{m-\ell} + c\sum_{j=\ell+1}^{m} \big(X_j-f(X_{j-1}, \theta_0+u_k/n)\big) b^{m-j} =\\
b^{\frac 1 2 (m-\ell)}\bigg(\widehat \Xi^k_\ell b^{\frac 1 2 (m-\ell)} +  c\sum_{j=\ell+1}^{\frac 1 2 (m+\ell)} \big(X_j-f(X_{j-1}, \theta_0+u_k/n)\big) b^{\frac 1 2 (m+\ell)-j}\bigg) +\\
+c\sum_{j=\frac 1 2 (m+\ell)+1}^{m} \big(X_j-f(X_{j-1}, \theta_0+u_k/n)\big) b^{m-j}=: b^{\frac 1 2 (m-\ell)}J^k_1+J^k_2, \quad \ell \le m.
\end{multline*}
Using the bound \eqref{momb}, we get
\begin{equation}\label{J1}
\E_{\theta_0} \big|J^k_1\big|^{2} \le
2\E_{\theta_0}|\widehat \Xi^k_\ell|^{2} +
C_1
\sum_{j=\ell+1}^{\frac 1 2 (m+\ell)} \E_{\theta_0}\Big(|X_j|+|X_{j-1}|\Big)^{2} |b|^{\frac 1 2 (m+\ell)-j}\le  C_2.
\end{equation}
By the triangle inequality
\begin{equation}\label{tre}
\begin{aligned}
 \E_{\theta_0}&\Big|\E_{\theta_0} \big(h(X_m,\xi_m, \widehat \Xi_m)\big|\F_\ell\big)-\widetilde \E_{\theta_0} h(X_0,\xi_0, \widehat \Xi_0)\Big| = \\
 \E_{\theta_0}&\Big|\E_{\theta_0} \big(h(X_m,\xi_m, \widehat \Xi_m)\big|\F_\ell\big)-\widetilde \E_{\theta_0} h(X_m,\xi_m, \widehat \Xi_m)\Big|\le \\
\E_{\theta_0}&\Big|\E_{\theta_0} \big(h(X_m,\xi_m,  b^{\frac 1 2 (m-\ell)}J_1+J_2)\big|\F_\ell\big)
-\E_{\theta_0} \big(h(X_m,\xi_m,  J_2)\big|\F_\ell\big)\Big| +\\
\E_{\theta_0}&\Big|
\E_{\theta_0} \big(h(X_m,\xi_m,  J_2)\big|\F_\ell\big)
-\widetilde \E_{\theta_0} h(X_m,\xi_m,  J_2)
\Big|+\\
&\Big|
\widetilde \E_{\theta_0} h(X_m,\xi_m,  J_2)
-\widetilde \E_{\theta_0} h(X_m,\xi_m,  b^{\frac 1 2 (m-\ell)}J_1+J_2)\Big|.
\end{aligned}
\end{equation}

Note that $h(X_m,\xi_m,  J_2)$ is measurable with respect to $\F_{\frac 1 2 (m+\ell),\infty}$ and  by \eqref{ge}
$$
\E_{\theta_0}\Big|
\E_{\theta_0} \big(h(X_m,\xi_m,  J_2)\big|\F_\ell\big)
-\widetilde \E_{\theta_0} h(X_m,\xi_m,  J_2)
\Big|\le C r^{\frac 1 2 (m-\ell)}.
$$
By the Lipschitz property of $h$ and \eqref{J1}, we have
\begin{align*}
&\E_{\theta_0}\Big|\E_{\theta_0} \big(h(X_m,\xi_m,  b^{\frac 1 2 (m-\ell)}J_1+J_2)\big|\F_\ell\big)
-\E_{\theta_0} \big(h(X_m,\xi_m,  J_2)\big|\F_\ell\big)\Big|\le \\
&\E_{\theta_0}\Big|h(X_m,\xi_m,  b^{\frac 1 2 (m-\ell)}J_1+J_2)
-h(X_m,\xi_m,  J_2)\Big|\le \\
& b^{\frac 1 2 (m-\ell)}\E_{\theta_0}L\Big(1+|X_m|+|\xi_m|+\|\widehat \Xi_m\|+\|J_2\|\Big)\|J_1\|\le \\
& b^{\frac 1 2 (m-\ell)}L\bigg(\E_{\theta_0}\Big(1+|X_m|+|\xi_m|+\|\widehat \Xi_m\|+\|J_2\|\Big)^2\bigg)^{1/2}\Big(\E_{\theta_0}\|J_1\|^2\Big)^{1/2}
\le C_3 b^{\frac 1 2 (m-\ell)}.
\end{align*}
Similar bound holds for the last term in \eqref{tre} and the  claim follows with $q:=\sqrt{|b|\vee r}$.
\end{proof}

\section*{Acknowledgement} P.Chigansky is supported by ISF grant 314/09



\begin{thebibliography}{10}

\bibitem{MH01}
M.~Caner and B.~E. Hansen.
\newblock Threshold autoregression with a unit root.
\newblock {\em Econometrica}, 69(6):1555--1596, 2001.

\bibitem{CMR05}
O.~Capp{\'e}, E.~Moulines, and T.~Ryd{\'e}n.
\newblock {\em Inference in hidden {M}arkov models}.
\newblock Springer Series in Statistics. Springer, New York, 2005.

\bibitem{Ch93}
K.~S. Chan.
\newblock Consistency and limiting distribution of the least squares estimator
  of a threshold autoregressive model.
\newblock {\em Ann. Statist.}, 21(1):520--533, 1993.

\bibitem{ChK10}
N.H. Chan and Yu.~A. Kutoyants.
\newblock Recent developments of threshold estimation for nonlinear time
  series.
\newblock {\em Journal of the Japan Statistical Society}, 40(2):277--308, 2010.

\bibitem{ChK11}
N.H. Chan and Yu.~A. Kutoyants.
\newblock On parameter estimations of threshold autoregressive models.
\newblock {\em Statistical Inference for Stochastic Processes}, 1:81--104,
  2012.

\bibitem{DD03}
J.~Dedeker and P.~Doukhan.
\newblock A new covariance inequality and applications.
\newblock {\em Stochastic Process. Appl.}, 106:63--80, 2003.

\bibitem{DL99}
P.~Doukhan and S.~Louhichi.
\newblock A new weak dependence condition and applications to moment
  inequalities.
\newblock {\em Stochastic Process. Appl.}, 84(2):313--342, 1999.

\bibitem{H11}
B.E. Hansen.
\newblock Threshold autoregression in economics.
\newblock {\em Statistics and Its Interface}, 4:123--127, 2011.

\bibitem{IH}
I.~A. Ibragimov and R.~Z. Has'minskii.
\newblock {\em Statistical Estimation: Asymptotic Theory}.
\newblock New York, 1981.

\bibitem{Kut10}
Yu.~A. Kutoyants.
\newblock On identification of the threshold diffusion processes.
\newblock {\em Annals of the Institute of Statistical Mathematics},
  64(2):383--413, 2012.

\bibitem{LT05}
Shiqing Ling and Howell Tong.
\newblock Testing for a linear {MA} model against threshold {MA} models.
\newblock {\em Ann. Statist.}, 33(6):2529--2552, 2005.

\bibitem{LTL07}
Shiqing Ling, Howell Tong, and Dong Li.
\newblock Ergodicity and invertibility of threshold moving-average models.
\newblock {\em Bernoulli}, 13(1):161--168, 2007.

\bibitem{LiSh2}
R.~S. Liptser and A.~N. Shiryaev.
\newblock {\em Statistics of random processes. {II}}, volume~6 of {\em
  Applications of Mathematics (New York)}.
\newblock Springer-Verlag, Berlin, 2 edition, 2001.

\bibitem{LLS11}
W.~Liu, S.~Ling, and Q-M. Shao.
\newblock On non-stationary threshold autoregressive models.
\newblock {\em Bernoulli}, 17(3):969--986, 2011.

\bibitem{M73}
R.~M. Meyer.
\newblock A {P}oisson-type limit theorem for mixing sequences of dependent
  ``rare'' events.
\newblock {\em Ann. Probability}, 1:480--483, 1973.

\bibitem{MT09}
S.~Meyn and R.~L. Tweedie.
\newblock {\em Markov chains and stochastic stability}.
\newblock Cambridge University Press, Cambridge, second edition, 2009.

\bibitem{PCT91}
D.~T. Pham, K.~S. Chan, and H.~Tong.
\newblock Strong consistency of the least squares estimator for a nonergodic
  threshold autoregressive model.
\newblock {\em Statist. Sinica}, 1(2):361--369, 1991.

\bibitem{Tong83}
H.~Tong.
\newblock {\em Threshold models in nonlinear time series analysis}, volume~21
  of {\em Lecture Notes in Statistics}.
\newblock Springer-Verlag, New York, 1983.

\bibitem{Tong11}
H.~Tong.
\newblock Threshold models in time series analysis - 30 years on.
\newblock {\em Statistics and Its Interface}, 4(2):107--118, 2011.

\bibitem{Tsay89}
R.~S. Tsay.
\newblock Testing and modeling threshold autoregressive processes.
\newblock {\em J. Amer. Statist. Assoc.}, 84(405):231--240, 1989.

\end{thebibliography}

\end{document}